%% file: llrs6.tex
\documentclass[]{amsart}

\usepackage[colorlinks=true,hyperindex]{hyperref}
\usepackage{wasysym}
\usepackage{amsmath}
\usepackage{amssymb}
\usepackage{amsthm}
\usepackage{tikz}
\usetikzlibrary{matrix,arrows,backgrounds,shapes.misc,shapes.geometric,patterns,calc,positioning}
\usetikzlibrary{calc,shapes}
\usepackage{wrapfig}
\usepackage{epsfig}  
\usepackage[margin=1.3in]{geometry}
\usepackage{color}	 %
\input{xy}
\xyoption{poly}
\xyoption{2cell}
\xyoption{all}

\usepackage{tikz,tikz-cd}
\usetikzlibrary{matrix,arrows,calc,shapes}
\usetikzlibrary{decorations.markings}
\usetikzlibrary{matrix}
\tikzset{->-/.style={decoration={
  markings,
  mark=at position #1 with {\arrow{stealth}}},postaction={decorate}}}
\newtheorem{thm}{Theorem}[section]

\newtheorem{lemma}[thm]{Lemma}
\newtheorem{corollary}[thm]{Corollary}
\newtheorem{conjecture}[thm]{Conjecture}

\newtheorem*{thm*}{Theorem}

\theoremstyle{definition}
\newtheorem{definition}[thm]{Definition}

\theoremstyle{remark}

\newtheorem{remark}[thm]{Remark}
\newtheorem{example}[thm]{Example}

\numberwithin{equation}{section}

\newcommand{\calf}{\mathcal{F}}
\newcommand{\cala}{\mathcal{A}}

\newcommand{\calg}{\mathcal{G}}
\newcommand{\calp}{\mathcal{P}}

\newcommand{\za}{\alpha}
\newcommand{\zb}{\beta}
\newcommand{\zd}{\delta}

\newcommand{\ze}{\epsilon}
\newcommand{\zg}{\gamma}

\DeclareMathOperator{\Brac}{Brac}

\newcommand{\rz}{(\mathbb{R}^2,\mathbb{Z}^2)}
\newcommand{\zgl}{\gamma_{AB}^L}
\newcommand{\zgr}{\gamma_{AB}^R}

\newcommand{\ma}[2]{m_{{#2},{#1}}}

\title{On the ordering of the Markov numbers}

\begin{document}
\begin{abstract}
The Markov numbers are the positive integers that appear in the solutions of the equation $x^2+y^2+z^2=3xyz$. These numbers are a classical subject in number theory and have important ramifications in hyperbolic geometry, algebraic geometry and combinatorics. 
 
 It is known that the Markov numbers can be labeled 
 %by rational numbers $\frac{p}{q}$ between 0 and 1, or equivalently, 
 by the lattice points $(q,p)$ in the first quadrant and below the diagonal whose coordinates are coprime. 
  In this paper, we consider the following question. Given two lattice points, can we say which of the associated Markov numbers is larger? A complete answer to this question would solve the uniqueness conjecture formulated by Frobenius in 1913. We give a partial answer in terms of the  slope of the line segment that connects the two lattice points. We prove that the Markov number with the greater $x$-coordinate is larger than the other if the slope is at least $-\frac{8}{7}$ and that it is smaller than the other if the slope is at most  $-\frac{5}{4}$. 

As a special case, namely when the slope is equal to 0 or 1, we obtain a proof of two conjectures from Aigner's book ``Markov's theorem and 100 years of the uniqueness conjecture''. 
\end{abstract}

\author{Kyungyong Lee}
\address{Department of Mathematics, 
University of Alabama, Tuscaloosa, AL 35487,
USA; and School of Mathematics, 
Korea Institute for Advanced Study, Seoul 02455, Republic of Korea}
\email{kyungyong.lee@ua.edu; klee1@kias.re.kr}
\author{Li Li}
\address{Department of Mathematics
and Statistics,
Oakland University, 
Rochester, MI 48309-4479, USA 
}
\email{li2345@oakland.edu}
\author{Michelle Rabideau}
\address{Department of Mathematics, University of Hartford,
West Hartford, CT 06117-1599}
\email{Rabideau@hartford.edu}
\author{Ralf Schiffler}
\address{Department of Mathematics, University of Connecticut, 
Storrs, CT 06269-3009, USA}
\email{schiffler@math.uconn.edu}
\thanks{The first author was supported by the University of Alabama, Korea Institute for Advanced Study, and the NSF grants DMS-1800207 and DMS-2042786,  and the fourth author was supported by the NSF grant DMS-1800860. }

\maketitle
%\tableofcontents
%%%%%%%%%%%%%%%%%%%%%%%%%%%%
%
%  SECTION
%
%%%%%%%%%%%%%%%%%%%%%%%%%%%%
\section{Introduction}
In 1879, Andrey Markov studied the equation
\begin{equation}
 \label{eq M} 
 x^2+y^2+z^2=3xyz,
\end{equation}
which is now known as the \emph{Markov equation}.  A positive integer solution $(m_1,m_2,m_3)$ of (\ref{eq M}) is called a \emph{Markov triple} and the integers that appear in the Markov triples are called \emph{Markov numbers}. 
 For example $(1,1,1),(1,1,2),(1,2,5),(1,5,13),(2,5,29)$ are Markov triples and $1,2,5,13,29$ are Markov numbers.
 
The Markov numbers are related to approximation theory. Given a real number $\za$, its Lagrange
number $L(\za)$ is defined as the supremum of all real numbers $L$ for which there exist infinitely many rational numbers $\frac{p}{q}$ such that 
$|\za-\frac{p}{q} | < \frac{1}{Lq^2}$. Thus the Lagrange number measures how well the real number $\za$ can be approximated by rational numbers.

The \emph{Lagrange spectrum} is defined as the set of all Lagrange numbers $L(\za)$, where $\za$ ranges over all irrational real numbers. Considering it as a subset of the real line, it is known that the Lagrange spectrum is discrete below 3, it is fractal between 3 and a number $F\approx 4.5278$ called the \emph{Freiman number}, and it is continuous above $F$.

 Markov proved in 1879 \cite{M} that the Lagrange spectrum below 3 is precisely the set of all 
 $\sqrt{9m^2-4} / m $, where $m$ ranges over all Markov numbers.

In 1913, Frobenius conjectured that, for every Markov number $m$, there exists a unique Markov triple in which $m$ is the largest number, see \cite{F}. This uniqueness conjecture is still open today. 
It has inspired a considerable amount of research and the Markov numbers have important ramifications in number theory, hyperbolic geometry, combinatorics and algebraic geometry. For an overview, we refer to the recent textbooks \cite{A,R}. This uniqueness conjecture also is equivalent to a conjecture saying that exceptional bundles on $\mathbb{P}^2$ are uniquely determined by their ranks up to shift and dual (see \cite{Rudakov} and \cite[Section 3]{Kuleshov}; we thank Michael Shapiro for pointing out to us the equivalence of the two conjectures).

\subsection{Aigner's conjectures}  The Markov numbers can be represented in a binary tree called the \emph{Markov tree}. This Markov tree is combinatorially equivalent to the Farey or Stern-Brocot tree of rational numbers. Thus there is a correspondence between $\mathbb{Q}_{[0,1]}$ and the Markov numbers by considering corresponding positions in these trees. We henceforth will refer to a Markov number as $m_{\frac{p}{q}}$ where $p<q$ for relatively prime positive integers $p$ and $q$.
For example,
$m_{\frac{1}{1}}=2,
m_{\frac{1}{2}}=5,
m_{\frac{1}{3}}=13,
m_{\frac{2}{3}}=29
$.

In his textbook \cite{A}, Martin Aigner formulates the following three conjectures on the ordering of the Markov numbers.

\begin{conjecture}
 \label{Aigner conjectures}\  
 
\begin{itemize}
\item[(a)] $m_{\frac{p}{q}} < m_{\frac{p}{q+i}}$  constant numerator conjecture,
\smallskip

\item[(b)] $m_{\frac{p}{q}} < m_{\frac{p+i}{q}}$ constant denominator conjecture,
\smallskip

\item[(c)] $m_{\frac{p}{q}} < m_{\frac{p-i}{q+i}}$ constant sum conjecture,
\end{itemize}
for all $i>0$ such that all the fraction in the index are reduced, positive and less than $1$.

\end{conjecture}

The third and fourth authors proved the constant numerator conjecture recently using continued fractions \cite{RS}. The denominator conjecture and constant sum conjecture follows from the main result below (see Corollary \ref{cor:Aigner conjectures hold}).

\subsection{Main results}
In this article, we  prove Conjecture \ref{Aigner conjectures}. Indeed, we prove a much stronger result that we explain next. 

 Let $F=\{(q,p)\in \mathbb{Z}^2 \mid 1\le p<q,\ \gcd(p,q)=1\}$ denote the domain of the Markov number function $(q,p)\mapsto m_{\frac{p}{q}}$. For any slope $a\in\mathbb{Q}$ and any $y$-intercept $b\in\mathbb{Q}$, define the function 
\[M_{a,b}\colon \{ (x,y)\in F\mid y=ax+b\} \longrightarrow \mathbb{Z},\  (x,y)\mapsto m_{\frac{y}{x}}.\] 
We say that the Markov numbers {\em increase with $x$}  along the line $y=ax+b$ if  \[M_{a,b}(x,y) < M_{a,b}(x',y') \] whenever $x<x'$. Similarly, 
we say  the Markov numbers {\em decrease with $x$}  along the line if
\[M_{a,b}(x,y) > M_{a,b}(x',y') \] whenever $x<x'$.
 We say that the Markov numbers are \emph{monotonic} along the line $y=ax+b$ if they increase with $x$ or decrease with $x$ along the line. 

%
%We say that the Markov numbers \emph{increase} with $x$ along the line $y=ax+b$ if 
%\[m_{\frac{y}{x}}< m_{\frac{y'}{x'}}\] for any two points $(x,y),(x',y')$ on the line such that  $x<x'$, $ 1 \leq y < x$, $ 1\leq y' < x'$ and $\gcd(x,y)= \gcd(x',y')=1$.
%Similarly, we say the Markov numbers \emph{decrease} with $x$ along the line $y=ax+b$ if 
%\[m_{\frac{y}{x}}> m_{\frac{y'}{x'}}\] for any two points $(x,y),(x',y')$ satisfying the same conditions as above. We say the Markov numbers are \emph{monotonic} along a line if they are increasing or decreasing along that line.
% on the line such that  $x<x'$,  and $\gcd(x,y)= \gcd(x',y')=1$.

With this terminology, the three conjectures above %constant denominator and the constant sum conjectures 
can be restated by saying that the Markov numbers increase with $x$ along any line of slope $a=0$, $a=\infty$  and $a=-1$, respectively.
\smallskip

We are now ready to state our main result.
\begin{thm}\label{main thm}
\  
\begin{itemize}
\item[(a)] The Markov numbers increase with $x$ along any line $y=ax+b$ with slope $a\ge -\frac{8}{7}= -1.142857\cdots$. 
%\smallskip

\item[(b)] The Markov numbers decrease with $x$ along any line $y=ax+b$ with slope $a\le -\frac{5}{4}=-1.25$.
%\smallskip

\item[(c)] There exist  lines of slope  $-\frac{6}{5}$ and $-\frac{7}{6}$  along which the Markov numbers are not monotonic. 
\end{itemize}

\end{thm}

\begin{corollary}\label{cor:Aigner conjectures hold}
 The conjectures \ref{Aigner conjectures} hold.
\end{corollary}
We also obtain the following result that strengthens both the constant numerator and the constant denominator conjecture.
\begin{thm}
\label{thm 2}
We have 
 \[ m_{\frac{p+i}{q+i}} \ge m_{\frac{p}{q+i}}+m_{\frac{p+i}{q}},\] whenever the indices are reduced fractions smaller than one. 
 \end{thm}

The order relations between the Markov numbers is illustrated in Figure \ref{figschema}. The figure shows a neighborhood of the point $(q,p)$ in the plane together with the lines of slope $ -\frac{8}{7}$ and $-\frac{5}{4}$, respectively. The lattice points in the green area correspond to Markov numbers that are strictly smaller than $m_\frac{p}{q}$ and the lattice points in the red area correspond to Markov numbers that are strictly larger than $m_\frac{p}{q}$. 
 In particular, if the uniqueness conjecture fails, then the second lattice point whose Markov number is equal to $m_\frac{p}{q}$ must lie in the grey area. So we have the following corollary of Theorem \ref{main thm}. 

\begin{corollary}
(a) If $m_\frac{p}{q}=m_\frac{p'}{q'}$ and $(p,q)\neq(p',q')$,  then $$ -\frac{5}{4}< \frac{p-p'}{q-q'}<-\frac{8}{7}$$

(b) Given  any Markov number $m_\frac{p}{q}$, there are at most $\lfloor\frac{5q+4p+44}{54}\rfloor$ pairs $(p',q')$ satisfying $$m_\frac{p'}{q'}=m_\frac{p}{q}, \quad (0<p'<q',\  \gcd(p',q')=1)$$
\end{corollary}
\begin{proof}
(a) follows immediately from Theorem \ref{main thm}.

For (b), denote $(p_i,q_i)$ ($1\le i\le \ell$, $1\le p_1\le p_2\le \cdots\le p_\ell$) the set of all pairs $(p',q')$ satisfying the condition in (b). For any $1\le i<\ell$, denote $a=p_{i+1}-p_i$, $b=q_i-q_{i+1}$. Then (a) implies $(a,b)\in\mathbb{Z}_{>0}^2$ and $8/7<a/b<5/4$. The smallest possible value of $a$ is $6$ (in which case $b=5$), so $p_{i+1}-p_{i}\ge 6$, thus $p_\ell\ge 1+6(\ell-1)$.  On the other hand, $p_\ell$ is strictly less than the $y$-coordinate of the intersection point of the two lines $y=x$ and $y-p=-\frac{5}{4}(x-q)$ (which has slope $-5/4$ and passes through the point $(q,p)$). Since the intersection is $(\frac{5q+4p}{9},\frac{5q+4p}{9})$, we have $\frac{5q+4p}{9}>p_\ell$. Then $\frac{5q+4p-1}{9}\ge  p_\ell\ge 1+6(\ell-1)$, which implies
$\ell\le \lfloor\frac{1}{6}(\frac{5q+4p-1}{9}-1)+1\rfloor=\lfloor\frac{5q+4p+44}{54}\rfloor$.
\end{proof}
Note that in the above proof we have not used the fact that $\gcd(p',q')=1$. So the corollary also holds if we replace $m_{\frac{p}{q}}$ and  $m_{\frac{p'}{q'}}$ by $m_{q,p}$ and $m_{q',p'}$ respectively, and without assuming $\gcd(p,q)=\gcd(p',q')=1$. 
\begin{figure}
\begin{center}
\huge\scalebox{0.5}{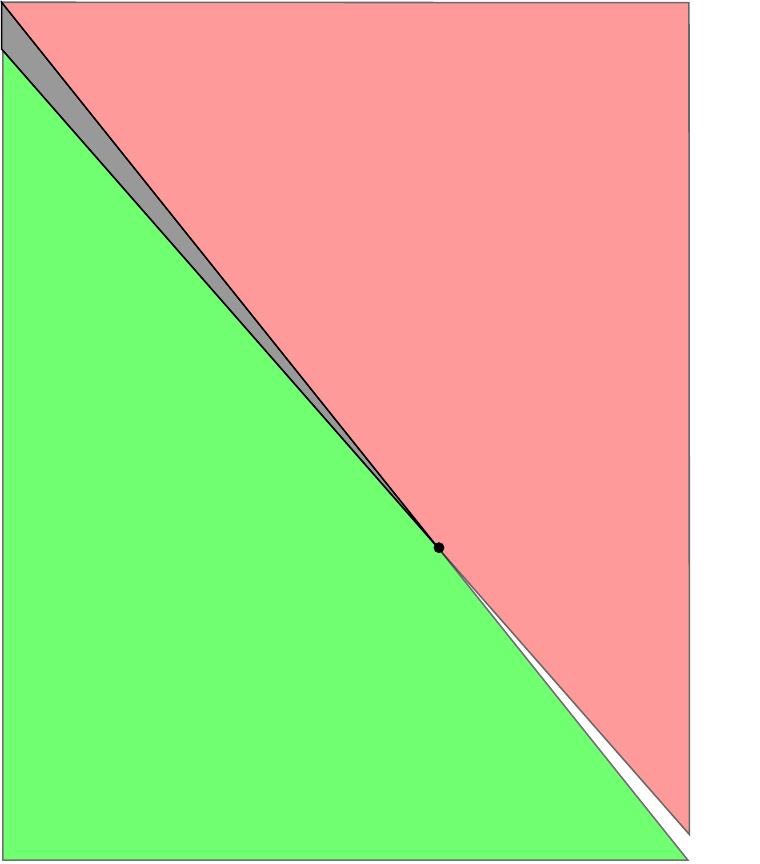}
\caption{Relations between Markov numbers. Markov numbers at lattice points in the green region are strictly smaller than $m_{\frac{p}{q}}$ and those at lattice  points in the red region are strictly larger than $m_{\frac{p}{q}}$. }
\label{figschema}
\end{center}
\end{figure}

\subsection{Ingredients for the proof}
The proof uses a connection to cluster algebras. It was observed in \cite{P,BBH} that the Markov numbers can be obtained from the cluster variables in the cluster algebra of the once-punctured torus by specializing the initial cluster variables to 1. Moreover, the clusters in the cluster algebra then specialize to the Markov triples.  On the other hand, the cluster variables can be computed by a combinatorial formula given as a summation over the perfect matchings of a so-called snake graph \cite{MSW}. 

Inspired by these results, we introduce a semi-metric, which we call the \emph{Markov distance}, that associates to a pair $(A,B)$ of lattice points in the plane an integer $|AB|$ defined as the number of perfect matchings of an associated snake graph. If the point $A$ is the origin and the point $B$ has coordinates $(q,p)$ with $\gcd(p,q)=1$ then the Markov distance $|AB|$ is equal to the Markov number $m_{\frac{p}{q}}$. This interpretation   allows us to define a number $m_{q,p}$ for pairs of integers that are not relatively prime. We prove all our inequalities in the more general setting of the numbers $m_{q,p}$.

We then use the skein relations, a result from cluster algebras proved in \cite{MW},  to show that the Markov distance satisfies some fundamental relations. 
Once these are established, our main result follows from elementary planar geometry arguments.

The paper is organized as follows. Section \ref{sect 1} reviews the connection to cluster algebras, the definition of snake graphs and the skein relations. In Section \ref{sect 2}, we introduce the Markov distance and deduce several key properties. We then give an elementary proof of Conjecture \ref{Aigner conjectures} and Theorem \ref{thm 2} in Section \ref{sect 3}.  Section \ref{sect main result} is devoted to the proof of our main result Theorem \ref{main thm}.
We end the paper with a comment on the Markov distance in Section \ref{sect 5}.

Some of the results of this paper have also been  proved recently in several papers using methods quite different from ours. In an independent and simultaneous work  \cite{LPTV}, C. Lagisquet, E. Pelantov\'a, S. Tavenas, and L. Vuillon proved Conjecture \ref{Aigner conjectures} using transformations of lattice paths. Shortly thereafter in \cite{McShane}, G. McShane gave a proof of Conjecture \ref{Aigner conjectures} using methods from hyperbolic geometry, namely a relationship between Markov numbers and the lengths of closed simple geodesics on the punctured torus. More recently in \cite{Gaster}, J. Gaster has proved the conjectures \ref{decreasing slope}, \ref{increasing slope}, and \ref{decreasing-increasing slope} for Markov numbers (where the full conjectures are for Markov distances); note that Gaster discovered those bounds independently since we did not put these conjectures in the first draft of the paper on arXiv.

\section{Relation to cluster algebras}\label{sect 1}

\subsection{Cluster algebras} In \cite{FST} the authors associate a cluster algebra to an arbitrary surface with marked points. In this paper we only need the cluster algebra of the torus with one puncture. 

Let $(S,p)$ be the torus with one puncture $p$  and let $\pi\colon\rz\to(S,p)$ be the universal cover. Define a triangulation $T$ of $\rz$ as follows, see Figure~\ref{fig11}. The vertices are the integer lattice points $\mathbb{Z}^2$, and the labeled edges come in the following three families;
\begin{enumerate}
\item horizontal edges are of the form $(i,j){\quad\over\quad}(i+1,j)$ and are labeled $1$;
\item vertical edges are of the form $(i,j){\quad\over\quad}(i, j+1)$ and are labeled $2$;
\item diagonal edges are of the form $(i,j){\quad\over\quad}(i+1,j-1)$ and are labeled $3$.
 \end{enumerate}
\begin{figure}
\begin{center}
 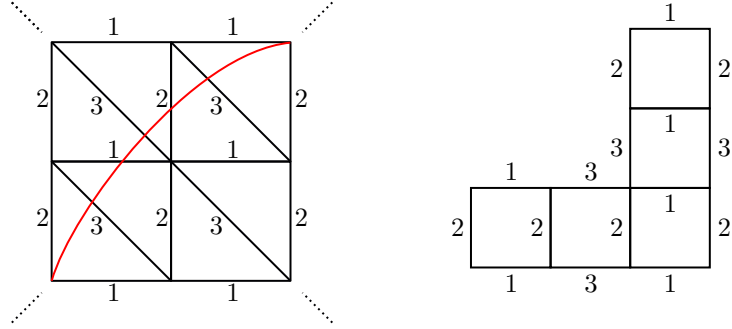
 \caption{On the left, the triangulation of the plane $\rz$ in black and an arc in red;  on the right,  the snake graph of the red arc.}
 \label{fig11}
\end{center}
\end{figure}

Via the covering map $\pi$, the triangulation $T$ induces a triangulation on the torus $(T,p)$. The quiver $Q$ of this triangulation is the following, see \cite{FST}.
\[\xymatrix{1\ar@<2pt>[rr]\ar@<-2pt>[rr] &&2\ar@<2pt>[ld]\ar@<-2pt>[ld]
\\ &3\ar@<-2pt>[lu]\ar@<2pt>[lu]
}\]
Let $\cala=\cala(Q)$ be the cluster algebra with trivial coefficients associated to this quiver. 

\subsubsection{Arcs} Arcs can be defined on an arbitrary surface with marked points. We shall only need them here in the plane $\rz$ with marked points the lattice points $\mathbb{Z}^2$, and in the torus with one puncture $(S,p)$ with marked point $p$.
 
 An \emph{arc} in a $\rz$ is a curve $\zg$ from a lattice point $A$ to a lattice point $B$ which does not pass through a third lattice point and does not cross itself. An \emph{generalized arc} in $\rz$ is a curve $\zg$ from a lattice point $A$ to a lattice point $B$ which does not pass through a third lattice point and has at most finitely many self-crossings. 
 
 An \emph{arc} in $(S,p)$ is a curve $\zg$ from $p$ to $p$ that does not visit $p$ except for its endpoints and that does not cross itself. It is well-known that the arcs on $(S,p)$ are precisely the images of line segments $O\!A$ in $\rz$ from the origin $O$ to a point $A=(q,p)$ such that $q,p$ are coprime integers. 
 
Arcs and generalized arcs are considered up to isotopy, where isotopies do not move curves over marked points.   In other words, two curves in $\rz$ represent the same arc if and only if they both start at the same lattice point, they both end at the same lattice point and the region bounded by the two curves does not contain any other lattice points.
We shall always assume that our (generalized) arcs are represented by curves that have a minimal number of crossing points with the triangulation $T$ as well as a minimal number of self-crossings. In particular. we assume that our generalized arcs do not contain any kinks (subcurves that are contractible loops). 

A \emph{closed loop} is a closed curve that is disjoint from all lattice points and has a finite number of self-crossings. A \emph{multicurve} is a finite multiset of generalized arcs and closed loops such that there are only finitely many  crossings among the curves. 

It was shown in \cite{FST} that the cluster variables in $\cala$ are in bijection with the arcs in $(S,p)$ and that the clusters are in bijection with triangulations.  In \cite{MSW} a 
 combinatorial formula was given for the cluster variables and in \cite{MSW2} this formula was used to associate cluster algebra elements to generalized arcs and closed loops as well as to multicurves.  If $C$ is a multicurve, we denote be $x_C$ the associated cluster algebra element.
 
 \subsubsection{Skein relations}
 Let $C$ be a multicurve and let $x$ be a crossing point in $C$. Thus $x$ could be a crossing point between two curves that are generalized arcs or closed loops, or $x$ could be a self-crossing point of a single generalized arc or a closed loop.    The \emph{smoothing} of $C$ at the point $x$ is defined as the pair of multicurves $C_+$ and 
$C_-$, where $C_+,C_-$ are the same as $C$ except for the local change that replaces the crossing $\times$ at the point $x$ with the pair of segments ${\asymp}$ or $\supset\subset$ respectively. Examples are shown in Figures \ref{fig 28case1}-\ref{fig 28case3-1}.

The following theorem is a special case of a result proved in \cite{MW} using hyperbolic geometry and in \cite{CS1,CS2,CS3} in a purely combinatorial way using snake graphs. 
\begin
{thm}[Skein relations]\label{skein relations}
 Let $C$ be a multicurve with a crossing point $x$ and let $C_-,C_+$ be the multicurves obtained by smoothing $C$ at $x$. Then we have the following identity in the cluster algebra associated to the surface.
 \[x_C=x_{C_+} + x_{C_-}\]
\end{thm}

%If $A$ and $B$ are lattice points in the plane, the \emph{line segment} from $A$ to $B$ is denoted by $\ell_{AB}$. Note that $\ell_{AB}$ is an arc if and only the coordinates of $B-A$ are relatively prime. 
%If $A$ is the origin and $B=(q,p)$, where $p,q$ are relatively prime integers, we use the notation $\ell_{\frac{p}{q}}$ for $\ell_{AB}$. Note that $\frac{p}{q}$ is the slope of $\ell_{\frac{p}{q}}$.

\subsubsection{Snake graphs}

Let $\zg$ be a generalized arc in $\rz$ from a point $A$ to a point $B$. Denote by $p_1,p_2,\ldots,p_n$ the crossing points of $\zg$ with $T$ in order along $\zg$. Let $i_j\in\{1,2,3\}$ be the label of the edge in $T$ that contains $p_j$.  The sequence $(i_1,i_2,\ldots, i_n)$ is called the \emph{crossing sequence} of $\zg$. It determines $\zg$ up to translation  and rotation by $180^\circ$. 
Note that $i_{j+1}\ne i_j$, for all $j=1,2,\ldots,n-1$, because $\zg$ has a minimal number of crossings with $T$. Moreover, since there are exactly three labels for the edges in $T$, we can define $b_j $ to be the unique label such that $\{i_j,i_{j+1},b_j\}=\{1,2,3\}$.

 To every (generalized) arc $\zg$ one can associate a planar graph $\calg(\zg)$ called the \emph{snake graph} of $\zg$ as follows. This construction was introduced in \cite{MS,MSW} for arcs on arbitrary surfaces. Here we give a simpler description adapted to our situation. 
For each $i=1,2,3$, we define two \emph{tiles}  $G_i^{\ze}$, with $\ze\in\{+,-\}$, as the labeled graphs shown in Figure \ref{fig tiles}.
\begin{figure}
\begin{center}
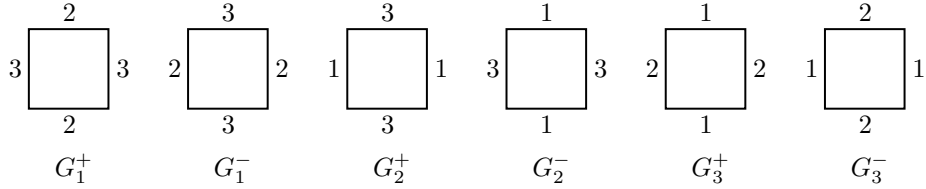
\caption{The six tiles that are the building blocks for the snake graphs}
\label{fig tiles}
\end{center}
\end{figure}
We remark that adding the diagonal from the northwest corner to the southeast corner in the tile $G_i^+$ and labeling this diagonal by $i$ will result in a graph that is a subgraph of the triangulation $T$. On the other hand, adding the same labeled diagonal to the tile $G_i^-$  will produce the mirror image of this subgraph. 

The snake graph $\calg(\zg)$ is defined recursively as follows. Let $i_1,i_2,\ldots,i_n$ be the crossing sequence of $\zg$. We start by laying down the tile $G_{i_1}^+$ exactly as it is shown in Figure \ref{fig tiles}. Next, we join the tile $G_{i_2}^{-} $ by identifying the unique edge labeled $b_1$ on the north or  east of $G_{i_1}^+$ to the unique edge labeled $b_1$ on the south or west of $G_{i_2}^-$, where we use the notation $b_1$ defined in the previous subsection. Note that this gluing edge is the north edge in $G_{i_1}^+$ if and only if is the south edge in $G_{i_2}^-$, and it is the east edge in $G_{i_1}^+$ if and only if is the west edge in $G_{i_2}^-$. 
Recursively, we join the tile $G_{i_{j+1}}^\ze$, with $\ze\in\{+,-\}$, to the graph consisting of $G_{i_1}^+,G_{i_2}^-,G_{i_3}^+,\ldots, G_{i_j}^{-\ze}$ by identifying the unique edge labeled $b_j$ on the north or  east of $G_{i_j}^{-\ze}$ to the unique edge labeled $b_j$ on the south or west of $G_{i_{j+1}}^\ze$.

For example, if $\zg$ is the arc shown in Figure \ref{fig11} then the crossing sequence is $3,1,3,2,3$ and the snake graph is built from the tiles $G_3^+,G_1^-,G_3^+,G_2^-,G_3^+$. This graph is shown in the right picture in Figure \ref{fig11}.

It is known that the unlabeled snake graphs are in bijection with continued fractions \cite{CS4}. This relation was crucial in the proof of part (a) of Conjecture \ref{Aigner conjectures} in \cite{RS}.
\subsubsection{Length function}
In order to introduce our length function we need the following concept from graph theory.
A \emph{perfect matching} of a graph is a subset of the set of edges  such that every vertex of the graph belongs to exactly one edge of the perfect matching.
\begin{definition}
 Let $\zg$ be a generalized arc in $\rz$. We define its \emph{length} $|\zg|$ to be the number of perfect matchings of its snake graph $\calg(\zg)$. 
\end{definition}
 Note that $|\zg|$ is a positive integer and  $|\zg|=1$ if and only if $\zg\in T$.

\subsubsection{Relation between the cluster algebra and the Markov numbers}
%The universal cover $\pi\colon \rz\to (S,p)$ from the plane to the once punctured torus maps the triangulation $T$ to a triangulation $T_S$ of $(S,p)$ which is compatible with our labeling of $T$.
% Let $\cala$ be the cluster algebra of $(S,p)$ with trivial coefficients as defined in \cite{FST}. 
% 
As we have seen above, the cluster variables in the cluster algebra $\cala$ of the once-punctured torus correspond to the line segments $\ell_{\frac{p}{q}}$  from $(0,0)$ to $(q,p)$ in $\rz$ with $\gcd(p,q)=1$. We denote the cluster variable of $\ell_{\frac{p}{q}}$ by $x_{\frac{p}{q}}$.  Each cluster variable is a Laurent polynomial in three variables $x_1,x_2,x_3$, which can be computed by a combinatorial formula as a sum over all perfect matchings of the snake graph  $\calg(\ell_{\frac{p}{q}})$ associated to $\ell_{\frac{p}{q}}$ \cite{MSW}. %Moreover, this snake graph is determined by a continued fraction $C(\ell_{\frac{p}{q}})$ whose numerator is equal to the number of perfect matchings \cite{CS4}.
 
 It is shown in \cite{BBH, P} that the specialization of the cluster variable $x_{\frac{p}{q}}$ at $x_1=x_2=x_3=1$ is equal to the Markov number $m_{\frac{p}{q}}$ of slope $\frac{p}{q}$, where $0<p<q$ and $\gcd(p,q)=1$. 
Therefore, the Markov number is equal to the number of perfect matchings of the snake graph $\calg({\ell_\frac{p}{q}})$. In terms of our length function, this can be restated as follows. 
\begin{equation}\label{eq markov=length}
m_{\frac{p}{q}}= |{\ell_\frac{p}{q}}|.
 \end{equation}
              Thus looking at the notation $m_{\frac{p}{q}}$ from another perspective, $\frac{p}{q}$ is not only the corresponding rational from the Farey tree, but it also refers to the slope of the line from the origin to $(q,p)$ from which a Markov snake graph can be constructed and the associated continued fraction yields the Markov number $m_{\frac{p}{q}}$ itself.
 In Section \ref{sect 2}, we shall generalize this correspondence to pairs of integers $(q,p)$ that are not necessarily coprime.

%In \cite{RS}, this correspondence was used to prove the constant numerator conjecture.

It follows from the above discussion that every skein relation $x_C=x_{C_-}\!+x_{C_+}$ induces an integer equation  $|x_C|=|x_{C_-}|+|x_{C_+}|$. In particular, we have the following special case.
\begin{corollary}[Generalized Ptolemy equality] \label{lem 11}
 Let $\za,\zb,\zd,\ze$ be arcs in $\rz$ that form a quadrilateral  $Q$ 
 that does not contain
  any lattice points besides the four endpoints of the arcs. Suppose that $\za,\zd$ are opposite sides, and $\zb,\ze$ are opposite sides of $Q$. Let $\zg_1$ and $\zg_2$ be the two diagonals of $Q$. Then the length function satisfies 
 \[| \zg_1|\,|\zg_2| = 
 |\za|\,|\zd| +|\zb|\,|\ze|. \]
 \end{corollary}

\subsubsection{Bracelets}

Let $\zeta$ be a closed simple curve. Following \cite{MSW2}, we let $\Brac_k\zeta$ be the $k$-fold concatenation of $\zeta $ with itself, see Figure~\ref{fig bracelet}. $\Brac_k\zeta$ is called the $k$-bracelet of $\zeta$. 
\begin{figure}
\begin{center}
\scalebox{0.5}{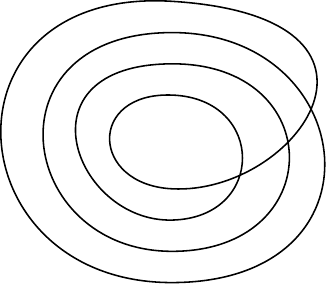}
\caption{A 4-bracelet}
\label{fig bracelet}
\end{center}
\end{figure}

\begin{lemma} Let  $\zeta$ be a closed simple curve in $\rz$ that is contractible to a lattice point.
Then $|\Brac_k \zeta|=2$, for all $k>0$.
\end{lemma}
\begin{proof}
This follows easily from the skein relations and induction.
%\footnote{Add Chebyshev recursion?}
\end{proof}

%%%%%%%%%%%%%%%%%%%%%%%%%%%%
%
%  SECTION
%
%%%%%%%%%%%%%%%%%%%%%%%%%%%%
\section{The Markov distance}\label{sect 2}
In this section we introduce a semi-metric called Markov distance between to lattice points as the number of perfect matchings of an associated snake graph. 

 Let $A$ and $B$ be two lattice points and $\ell_{AB}$ the line segment from $A$ to $B$. Denote by $(s,r)$ the coordinates of $B-A$. Thus $\ell_{AB}$ has slope $\frac{r}{s}$. 
If $\gcd(r,s)=1$, then there are no lattice points on the line segment $\ell_{AB}$ besides $A$ and $B$, and in this case, the line segment $\ell_{AB}$ is an arc. Otherwise, we let $p,q $ be relatively prime integers such that  $\frac{r}{s}=\frac{p}{q}$ and denote by
 $P_0=A,P_1,P_2,\ldots,P_t=B$ the sequence of lattice points along $\ell_{AB}$. Thus $t=\gcd(r,s)$ and  $P_i=A+i(q,p)$, for $i=1,2,\ldots,t$.
\begin{definition}\label{def:left and right deformation}
With the notation above, we define the \emph{left deformation} $\zg_{AB}^L$ of the line segment $\ell_{AB}$
to be 
the arc  from $A$ to $B$ that is an infinitesimal deformation of $\ell_{AB}$ passing on the left of the points $P_1,P_2,\ldots,P_t$ with respect to the orientation from $A$ to $B$, see Figure~\ref{fig 21}.  

Similarly, the   \emph{right deformation} $\zg_{AB}^R$ of $\ell_{AB}$ 
is the arc  from $A$ to $B$ that is an infinitesimal deformation of the line segment $\ell_{AB}$ passing on the right of the points $P_1,P_2,\ldots,P_t$.\end{definition}

In the case where $\ell_{AB}$ is already an arc, we have $\zg^L_{AB}=\zg^R_{AB}=\ell_{AB}$.

\begin{lemma}
 \label{lem 21}
 Let $A$ and $B$ be two lattice points in the plane and let $\zgl$ and $\zgr$ be the left and right deformations of the line segment $\ell_{AB}$. Then both deformations have the same length, thus
 \[|\zgl|=|\zgr|.\]
\end{lemma}
\begin{figure}
\begin{center}
 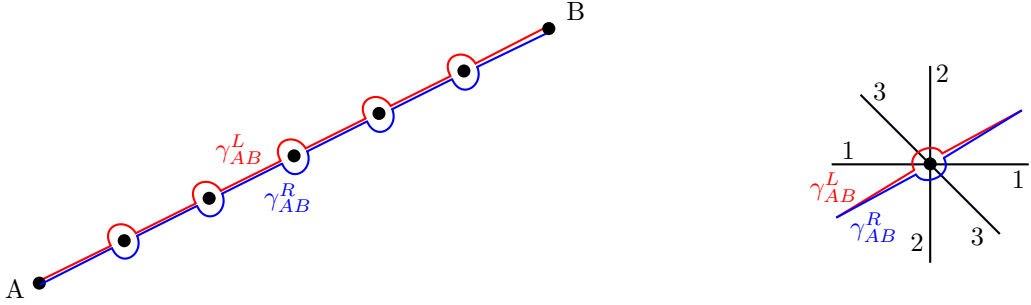
 \caption{The arcs $\zgl$ and $\zgr$ (left) and their local crossing pattern with the triangulation in the vicinity of a lattice point (right).}\label{fig 21}
\end{center}
\end{figure}
\begin{proof}
 The line segment $\ell_{AB}$ can be decomposed as the concatenation of line segments $\ell_{AB}=\ell_{AP_1}\ell_{P_1P_2}\ell_{P_2P_3}\cdots\ell_{P_{t-1}B}$. Each of the line segments $\ell_{P_iP_{i+1}}$ is an arc with the same crossing sequence $(i_1,i_2,\ldots,i_n)$. Moreover, since the crossing sequences of line segments are palindromic, we have $(i_1,i_2,\ldots,i_n)=(i_n,i_{n-1},\ldots,i_1)$.
 
 Every lattice point $P_i$ is incident to six edges of the triangulation, see the right picture in Figure~\ref{fig 21}. The arc $\zgl$ crosses the three edges labeled $1,3,2$ in order on the left of $P_i$ and the arc $\zgr$ crosses the edges labeled $2,3,1$ in order on the right of $P_i$. Thus the crossing sequence of $\zgl$ is $(i_1,i_2,\ldots,i_n,1,3,2,i_1,i_2,\ldots,i_n,1,3,2,\ldots,i_1,i_2,\ldots,i_n)$ and the  crossing sequence of $\zgr$ is $(i_1,i_2,\ldots,i_n,2,3,1,i_1,i_2,\ldots,i_n,2,3,1,\ldots,i_1,i_2,\ldots,i_n)$. Note that the second sequence is the reverse of the first, since $(i_1,i_2,\ldots,i_n)=(i_n,i_{n-1},\ldots,i_1)$.
 This implies that the snake graph of $\zgl$ is obtained from the snake graph of $\zgr$ by a rotation of  $180^\circ$ \cite[Proposition 3.1]{CS5}. In particular, the two snake graphs have the same number of perfect matchings, thus $|\zgl|=|\zgr|.$
\end{proof}
The lemma motivates the following definition. 
\begin{definition}\label{def distance}
The \emph{Markov distance} $|AB|$ between two lattice points $A,B$ in the plane is the semi-metric defined by the integer \[|AB|\,=\,|\zg^L_{AB}|.\]
We also set $|AA|=0$.

The Markov distance between the origin $O=(0,0)$ and a point $A=(q,p)$ will be denoted by 
\[m_{q,p}=|OA|. \]
\end{definition}

\begin{remark}
 Thus if $q,p$ are coprime then $m_{q,p}=m_{\frac{p}{q}}$ is the Markov number of slope $\frac{p}{q}$. 
\end{remark}

\smallskip

We are now ready for the main result of this section. It can be thought of as the analogue of the statement that the straight line is the shortest curve between two points in the plane, when replacing the Euclidean distance by the Markov distance.
This result is the key to the proof of the conjectures. 
Recall that a generalized arc is allowed to cross itself finitely many times \cite{MW}.
 \begin{thm}\label{thm 1}
 Let $A$ and $B$ be two lattice points in the plane and $\zgl$ the left deformation of the line segment $\ell_{AB}$.  Let $\zg$ be a generalized arc from $A$ to $B$. Then
 \[ |AB|\,\le \, |\zg|.\]
 and equality holds if and only if $\zg$ is homotopic to $\zgl$ or $\zgr$. (In the special case when $A=B$, the equality holds if and only if $\zg$ is homotopic to a contractible curve starting and ending in $A$.)
\end{thm}
\begin{proof}
 Let $\zg$ be  a generalized arc from $A$ to $B$  and such that $|\zg|$ is minimal among all arcs from $A$ to $B$. We want to show that $\zg=\zg_{AB}^L$ or $\zg=\zg_{AB}^R$. 

We start by showing that $\zg$ has no self-crossing. Suppose to  the contrary that $\zg$ does have a self-crossing at a point $x$. Then we can write $\zg$ as a concatenation of subcurves $\zg=\zg_1\zg_2\zg_3$, where $\zg_1$ runs from $A$ to the crossing point $x$, $\zg_2$ is a non-contractible closed curve starting and ending at $x$ and $\zg_3$ runs from $x $ to $B$. The skein relation (Theorem \ref{skein relations}) applied to the smoothing of the crossing at $x$ implies
\[|\zg| =|\zg_1\zg_3| \,| \zg_2| + |\ze|  \] where $\ze$ is the generalized arc obtained in the smoothing. Note that, since $\zg$ has no kink,  $\ze$ also does not contain a kink and thus $|\ze|>0$. Moreover $|\zg_2|>0$, since $\zg_2$ is non-contractible. Thus \[|\zg| >|\zg_1\zg_3|\] with $\zg_1\zg_3 $ a generalized arc from $A$ to $B$. This contradicts the minimality of $|\zg|$. 

Thus $\zg$ is a simple curve in the plane from $A$ to $B$ that does not meet any other lattice points besides $A$ and $B$. We may think of the plane as a board with a peg in each lattice point and $\zg$ as a string lying on this board. If we now pull this string taut it will touch a sequence of pegs in the board. This image shows that there exists a unique sequence of lattice points $A=Q_0,Q_1,\ldots,Q_m=B$ each equipped with a circle $C_i$ of infinitesimally small radius and centered at $Q_i$, $i=1,2,\ldots, m-1$ such that $\zg$ is homotopic to the path
\[l_1a_1l_2a_2\cdots l_{m-1}a_{m-1}l_m\] 
defined as follows, see Figure \ref{fig piecewise path}.
\begin{figure}
\begin{center}
\scalebox{1.2}{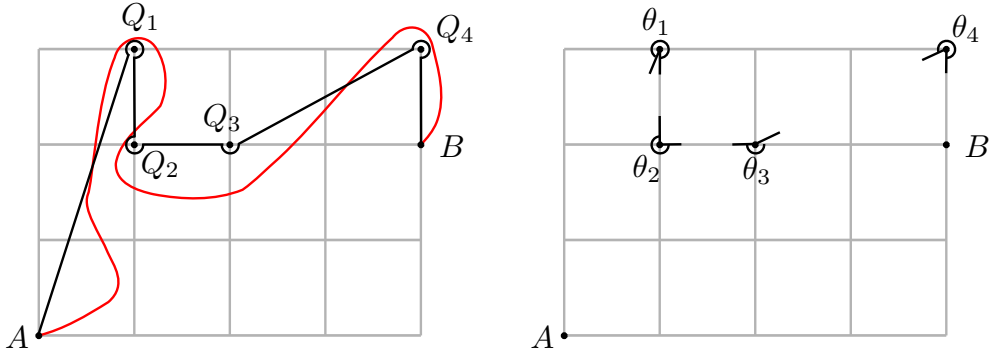}
\caption{Proof of Theorem \ref{thm 1}. The construction of the path and the angles. 
The left picture shows an arc $\zg$ in red and the associated path $l_1a_1\ldots,l_m$ in black. The right picture shows the angles $\theta_i$. Here $\theta_1,\theta_4\in [-2\pi,-\pi]$ and  $\theta_2,\theta_3\in [\pi,2\pi]$}
\label{fig piecewise path}
\end{center}
\end{figure}
\begin{itemize}
\item $l_1 $ is the straight line segment that starts at $Q_0=A$ and ends at a point $R_1$ on the circle $C_1$,
and
$a_1$ is an arc along the circle $C_1$ from $R_1$ to a point $S_1$;
\item 
for $i=2,\ldots m-1$, $l_i$  is the straight line segment that starts at $S_{i-1}$ on the circle $C_{i-1} $ and ends at a point $R_{i}$ on the circle $C_i$,
and  
$a_i$ is an arc along the circle $C_i$ from $R_i$ to a point $S_i$;
\item $l_m$ is  the straight line segment that starts at $S_{m-1}$ on the circle $C_{i-1} $ and ends at $Q_m=B$. 
\end{itemize}

We define a sequence of angles $\theta_1,\theta_2,\ldots,\theta_{m-1}$ by setting $\theta_i$ equal to the angle between the line segments $Q_{i-1}Q_i$ and $Q_iQ_{i+1}$ following the curve $\zg$.  It follows that $\theta_i\in [-2\pi,-\pi] \cup [\pi, 2\pi]$, indeed, if $-\pi<\theta<\pi$ then the point $Q_i$ would  not be a point in the above sequence.

With this notation, the arc $\zg$ is equal to the arc $\zg_{AB}^L$, or  $\zg_{AB}^R$ respectively, if and only if all angles $\theta_i$ are equal to $\pi$, or $-\pi $ respectively.
Assume $\zg$ is not equal to  $\zg_{AB}^L$ or  $\zg_{AB}^R$.  Without loss of generality, we may assume $\theta_1>0$.  Let $\theta _{s+1}$ be the first angle that is not equal to $\pi$. We consider three cases.

(1) Suppose $\theta_{s+1}\in[-2\pi,-\pi]$. Since $\theta_1=\theta_2=\ldots=\theta_s=\pi$, there is a pair $(p,q)$ of relatively prime integers such that each vector $\overrightarrow{Q_iQ_{i+1}}=\left(
\begin{smallmatrix}
 p\\q
\end{smallmatrix}
\right)
$, for $i=1,2, \ldots, s$. Let $D=Q_s+(p,q)$, see Figure \ref{fig 28case1}, and consider the arc $\zg_{AD}^L$. Note that $D=Q_{s+1}$, because the angle $\theta_s=\pi$. Applying the skein relations (Theorem \ref{skein relations}) to the product of $\zg$ and $\zg_{AD}^L$ by smoothing their crossing near $D$, we obtain
\[|\zg_{AD}^L|\,|\zg| = |\zg_1|\,|\zg_{AD}^R|+ |\ze|\,|\zb|,\]
where $\zg_1, \zb,\ze$ are the arcs obtained from the smoothing as shown in Figure \ref{fig 28case1}.
%where $\zg_1 $ is obtained from $\zg$ by changing the angles $\theta_i$ from $\pi $ to $-\pi$ for $i=1,2,\ldots, s$ and $\ze$ is the arc from $A$ to $A$ obtained by following $\zg_{AQ_s}^L$ but going around the point $Q_s$ and then following $\zg_{Q_sA}^L$, whereas $\zg$ is the arc from $D$ to $B$ following $\zg$.
Lemma \ref{lem 21} implies $|\zg_{AD}^L|=|\zg_{AD}^R|$ and thus $|\zg| > |\zg_1|$, a contradiction to the minimality of $|\zg|$.
\begin{figure}
 \begin{center}
\scalebox{0.7}{ 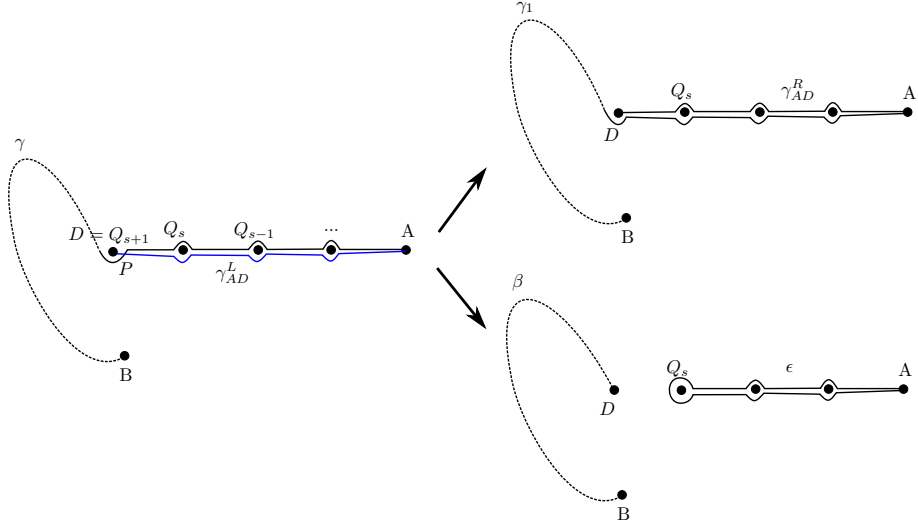}
 \caption{Proof of Theorem \ref{thm 1}. Step 3, Case 1.}\label{fig 28case1}
\end{center}
\end{figure}

(2) Suppose $\theta_{s+1}\in[\pi,2\pi)$.
Again there is a pair $(p,q) $ of relatively prime integers such that each vector $\overrightarrow{Q_iQ_{i+1}}=\left(
\begin{smallmatrix}
 p\\q
\end{smallmatrix}
\right)$, for $i=1,2, \ldots, s$, and we let $D=Q_{s+1}=Q_s+(p,q)$, see Figure \ref{fig 28case2}.
Let $E=D+(p,q)$. Note that $E$ is not one of the points $Q_i$. Applying the skein relation to the product of $\zg$ and $\zg_{AE}^L$ by smoothing their crossing near $D$, we obtain
\[|\zg_{AE}^L|\,|\zg| = |\zg_1|\,|\zg_{AE}^R|+ |\ze|\,|\zb|,\]
%where
%$\zg_1$ is obtained from $\zg$ by changing the angles $\theta_1,\ldots,\theta_s$ from $\pi$ to $-\pi$ and removing the point $D=Q_{s+1}$ from the sequence, and $\ze$ is the arc from $A$ to $A$ obtained by following $\zg_{AD}^L$ but going around the point $D$ and then following $\zg_{DA}^L$, whereas $\zb$ is the arc from $E$ to $B$ following the line segment from $E$ to $D$ up to the crossing point and then following $\zg$ up to $B$.
where $\zg_1, \zb,\ze$ are the arcs obtained from the smoothing as shown in Figure \ref{fig 28case2}.  Again Lemma~\ref{lem 21} allows us to conclude $|\zg|>|\zg_1|$, which is a contradiction to the minimality of $|\zg|$.
\begin{figure}
 \begin{center}
\scalebox{0.7}{ 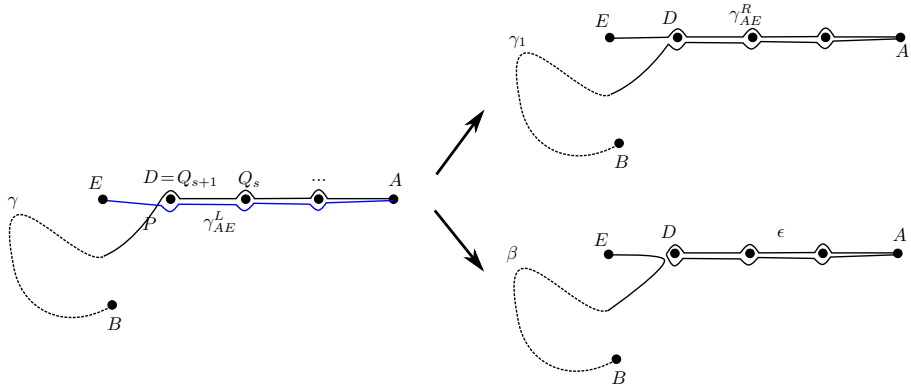}
 \caption{Proof of Theorem \ref{thm 1} Case 2.}\label{fig 28case2}
\end{center}
\end{figure}

(3) Suppose $\theta_{s+1}=2\pi$. In this case $Q_s=Q_{s+2}$. Let $D=Q_{s+1}$, see Figure \ref{fig 28case3-1}.
If all subsequent angles are equal to $\pi$ then $B$ lies on the line through $A$ and $D$ and  $\zg$ is obtained from the concatenation of $\zg_{AD}^R\zg_{DB}^R$ by an infinitesimal deformation avoiding the point $D$ on the right. Clearly $|\zg| > \zg_{AB}^L$ in this case, a contradiction.

Suppose therefore that at least one of the angles after $\theta_{s+1}$  is different from $\pi$. 
Let $t >0$ be the least integer such that 
 $\theta_{s+1+t}\ne \pi$.
 
 If $t<s$ then we must have $\theta_{s+1+t}\in[-2\pi,-\pi]$ since $\zg$ has no self-crossing. This case is illustrated in Figure \ref{fig 28case3-1}. 
Similarly to the case (1), applying the skein relation to the product of $\zg$ and $\zg_{AD}^L$ by smoothing their crossing near $D$, we obtain
\[|\zg_{AD}^L|\,|\zg| = |\zg_1|\,|\zg_{AD}^R|+ |\ze|\,|\zb|,\]
where $\zg_1,\ze,\zb$ are the arcs obtained in the smoothing, see Figure \ref{fig 28case3-1}. Again we conclude $|\zg|>|\zg_1|$, which is a contradiction to the minimality of $|\zg|$.

We do not illustrate the case $t>s$ here, but it also leads to a contradiction using an argument similar to the one in case (2) above.
Thus each of the cases (1)-(3)  leads to a contradiction, and therefore we must have that all angles $\theta_i$ are equal to $\pi$. Thus $\zg=\zg_{AB}^L$. \qedhere
\begin{figure}
 \begin{center}
\scalebox{0.7}{ 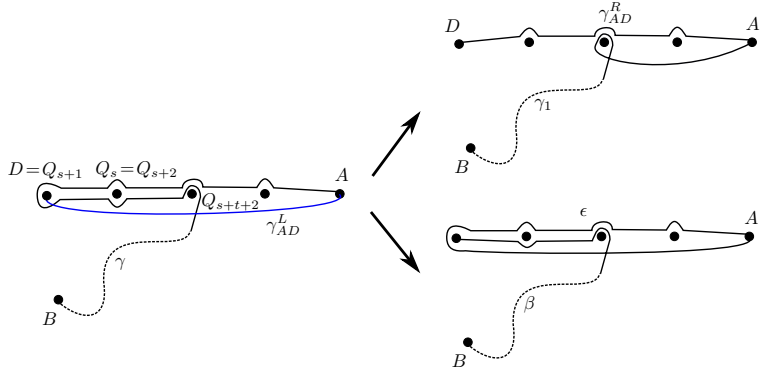}
 \caption{Proof of Theorem \ref{thm 1}. Case 3. with $t<s$.}\label{fig 28case3-1}
\end{center}
\end{figure}
%\begin{figure}
% \begin{center}
%\scalebox{0.7}{ \input{figure28case3_2AB.pdf_tex}}
% \caption{Proof of Theorem \ref{thm 1'}. Step 3, Case 3-2.}\label{fig 28case3-2}
%\end{center}
%\end{figure}

\end{proof}
As an immediate consequence, we obtain the following Ptolemy inequality.
\begin{corollary}[Ptolemy inequality]\label{Ptolemy inequality}
 Given four points $A,B,C,D$ in the plane such that the straight line segments $\ell_{AB},\ell_{BC},\ell_{CD},\ell_{DA}$ form a convex quadrilateral with diagonals $\ell_{AC}$ and $\ell_{BD}$, we have
 \[|AC|\,|BD|\  \ge\  |AB|\,|CD| \, + \, |AD| \,|BC|.\]
\end{corollary}
\begin{proof}
 Consider the arcs $\zg_{AC}^L$ and $\zg_{BD}^L$ given as the left deformations of the diagonals. Let $\za_{AB}$ be the arc from $A$ to $B$ defined by following $\zg_{AC}^L$ up to its crossing point with $\zg_{BD}^L$ and then following $\zg_{BD}^L$ up to $B$. Similarly, we can define arcs $\za_{BC},\za_{C\!D}$ and $ \za_{D\!A}$ to obtain a curved quadrilateral with corner vertices $A,B,C,D$ that does not contain any other lattice points. Then Corollary \ref{lem 11} implies
 \[  |\zg_{AC}^L |\, |\zg_{B\!D}^L| \, =\,|\za_{A\!B}|\,|\za_{C\!D}|\, + \,|\za_{BC}|\,|\za_{D\!A}|\,.\]
 The left hand side of this equation is equal to $|AC|\,|BD|$, by definition of the Markov distance, and thus the result follows from Theorem \ref{thm 1}. 
\end{proof}

\begin{remark}
 The Markov distance is a semi-metric but not a metric. It does not satisfy the triangle inequality because the deformations may create too many crossing points. For example, take the three colinear points $O=(0,0), A=(1,0), B=(2,0)$. Then $|OA|=|AB|=1$ but $|OB|=3$, because the deformation $\zg^L_{OB}$ crosses two arcs of the triangulation. For a non-colinear example, take $ D=(3,0), C=(1,1)$. Then each of the line segments $\ell_{OC}, \ell_{CD}$ crosses exactly one arc of the triangulation, hence $|OC|= |CD|=2$. However $|OD|=8$. 
 
However, there is a variation of the triangle inequality which we shall prove in Section \ref{sect 5}.
\end{remark}

%%%%%%%%%%%%%%%%%%%%%%%%%%%
%
 % SECTION
%
%%%%%%%%%%%%%%%%%%%%%%%%%%%
\section{Proof of Aigner's conjectures}\label{sect 3}
We shall now prove Conjecture \ref{Aigner conjectures} from Aigner's book in the following theorem. We actually prove more general statements in Theorem \ref{thm 1new} about the numbers $m_{q,p} $ introduced in Definition \ref{def distance} for arbitrary lattice point $0\le p\le q$.
Recall that $m_{q,p}=m_{\frac{p}{q}}$ is a Markov number if  $0<p<q$ and $p,q$ are relatively prime. 

Note that the theorem is proved independently by C. Lagisquet, E. Pelantov\'a, S. Tavenas, L. Vuillon in \cite{LPTV} using a different method.
\begin{thm} \label{thm 1new}
 For all integers $0\le p\le q$\,, we have the following inequalities.
\begin{align}
m_{q,p}&<m_{q,p+1},\label{constant denominator} \\
 m_{q,p}&<m_{q+1,p}, \label{constant numerator}\\
m_{q,p}&<m_{q+1,p-1},\quad (\textup{if }p\ge 1).\label{constant sum}
\end{align}
\end{thm}

\begin{proof}
 Let $A=(0,0), B=(q,p)$, $C=(q+1,p)$ and $ D=(q+1,p-1)$, see Figure \ref{fig conjectures}.
\begin{figure}
\begin{center}
\scalebox{0.8}{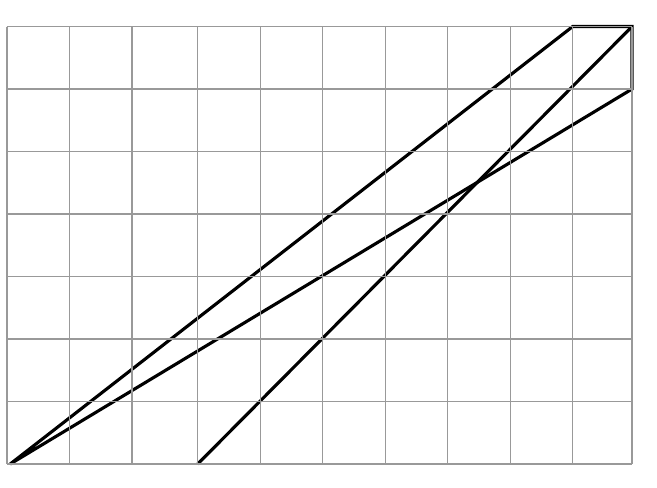}
\caption{Proof of Theorem \ref{thm 1new}}
\label{fig conjectures}
\end{center}
\end{figure}
 The Ptolemy inequality (Corollary \ref{Ptolemy inequality}) implies
 \[|AC|\,|BD|\  \ge\  |AB|\,|CD| \, + \, |AD| \,|BC|.\]
But $|BD|=|CD|=|BC|=1$, since the corresponding line segments are  arcs of the triangulation $T$. So 
\begin{equation}
\label{strong eq} 
 |AC|\ge|AB|+|AD|.
\end{equation}
In particular, we have two strict inequalities $
 |AC|> |AB|$  
 and $ |AC|> |AD|$, and the former proves (\ref{constant numerator})  while the latter proves  (\ref{constant denominator}).

%It remains to show $|AD|>|AB|$.
% Let $L$ be the line of slope 1 that passes through the point $C=(q+1,p)$. Thus $L$ is the the bisector of the segment $BD$ and its equation is  $y=x+p-q-1$.  
In order to show (\ref{constant sum}), let $E=( q-p+1, 0)$, see Figure \ref{fig conjectures}. Then the coordinates of $B-E$ and $D-E$ are $(p-1,p)$ and $(p,p-1)$, respectively, and therefore
%{the reason is that the crossing sequences are $(3,2,3,1,3,2,3,1,\ldots,1,3,2,3)$ for $\ell_{\frac{p-1}{p}}$ and $(3,1,3,2,3,1,3,2,\ldots,2,3,1,3)$ for
%$\ell_{\frac{p}{p-1}}$ and thus the snake graphs are isomorphic.} 
we have $|BE|=|DE|$. 
 
 Consider the quadrilateral defined by the points $A,B,D,E$. The Ptolemy inequality  implies
 
\begin{equation} \label{eq 35}
 |AD| \, |BE| \ge |AB|\,|DE| + |BD|\,|AE|,
\end{equation}
 hence
 \[|AD| \, |BE| > |AB|\,|DE|,  \]
 and thus the inequality (\ref{constant sum}) follows since $|BE|=|DE|.$
% \[|OB|   > |OA| \]
\end{proof}

\begin{corollary}
 \label{cor conjectures}
 The conjectures \ref{Aigner conjectures} hold.
 \qed\end{corollary}
From our proof, we also obtain the following stronger result.
\begin{thm}\label{thm 2'}
 For all integers $0\le p\le q$\,, we have 
 \[ m_{q+1,p}\ge m_{q,p}+m_{q+1,p-1}.\]
 \end{thm}
\begin{proof}
This follows directly from equation (\ref{strong eq}). 
\end{proof}
\begin{example}
 $433=m_{5,3}>m_{4,3}+m_{5,2}=169+194=363.$
\end{example}
If we use this result repeatedly
\[\begin{array}{rcl} m_{q+i,p+i} &\ge& m_{q+i-1,p+i}+m_{q+i,p+i-1} \\
&\ge&
m_{q+i-2,p+i}+2m_{q+i-1,p+i-1}+m_{q+i,p+i-2} \\
&\ge&
m_{q+i-3,p+i}+3m_{q+i-2,p+i-1}+3m_{q+i-1,p+i-2} +m_{q+i,p+i-3} \\&\vdots&\\ 
%& \ge & m_{q,p+i}+{i\choose 1}m_{q+1,p+i-1}+{i\choose 2}m_{q+2,p+i-2} +\cdots +m_{q+i,p}\\
\end{array} \]
we have proved the following corollary. 
\begin{corollary}
 For all integers $0\le p\le q$ and $i>0$ we have
 \[ m_{q+i,p+i} \ge \sum_{j=0}^i {i\choose j} m_{q+j,p+i-j}.\]
In particular, Theorem \ref{thm 2} holds.
\end{corollary}

\section{The main result}\label{sect main result}
In this section we study the monotonicity of the Markov numbers on lines of a given slope. To be more precise, let $F=\{(q,p)\in \mathbb{Z}^2 \mid 1\le p<q,\ \gcd(p,q)=1\}$ denote the domain of the Markov number function $(q,p)\mapsto m_{\frac{p}{q}}$. For any slope $a\in\mathbb{Q}$ and any $y$-intercept $b\in\mathbb{Q}$, define the function 
\[M_{a,b}\colon \{ (x,y)\in F\mid y=ax+b\} \longrightarrow \mathbb{Z},\  (x,y)\mapsto m_{\frac{y}{x}}.\] 
We say that the Markov numbers {\em increase with $x$}, respectively {\em decrease with $x$},  along the line $y=ax+b$ if  $M_{a,b}(x,y) < M_{a,b}(x',y') $ whenever $x<x'$, respectively if  $M_{a,b}(x,y) > M_{a,b}(x',y') $ whenever $x<x'$. We say that the Markov numbers are monotonic along the line $y=ax+b$ if they increase with $x$ or decrease with $x$ along the line.

  It seems that there are critical slopes $a_1<a_2<-1$ such that 
\begin{itemize}
\item[(i)] the  Markov numbers increase with $x$ along any line of slope $a\ge a_2$; 
\item[(ii)] the Markov numbers decrease with $x$  along any line of slope $a\le a_1$;
\item[(iii)] for every $a$ with $a_1< a<a_2$, there exists a line of slope $a$ along which the Markov numbers are not monotonic. 
\end{itemize}
 At this time, we do not know what the numbers $a_1$ and $a_2$ should be, but the following theorem settles 
 the question for slopes of the form $a=-\frac{n+1}{n}$. 

\begin{thm}[Theorem \ref{main thm}]
\  
\begin{itemize}
\item[(a)] The Markov numbers increase with $x$ along any line $y=ax+b$ with slope $a\ge -\frac{8}{7}$. 
\smallskip

\item[(b)] The Markov numbers decrease with $x$ along any line $y=ax+b$ with slope $a\le -\frac{5}{4}$.
\smallskip

\item[(c)] There exist  lines of slope  $-\frac{6}{5}$ and $-\frac{7}{6}$  along which the Markov numbers are not monotonic. 
\end{itemize}

\end{thm}

\begin{remark} It is not hard to see that the three conjectures proved in Theorem \ref{thm 1new} imply that the Markov numbers increase with $x$ on all lines of slope $a\ge -1$. Thus part (a) of Theorem \ref{main thm} sharpens the results of Theorem \ref{thm 1new}. 
\end{remark}

\begin{remark}
  Part (c) of the theorem is proved by the following examples. 
\[
\begin{array}{ccl}
%m_{\frac{7}{8}} < m_{\frac{1}{13}}  & \textup{on } y=-\frac{6}{5}x+\frac{83}{5}   \\[5pt]
 m_{\frac{13}{14}} > m_{\frac{7}{19}}  < m_{\frac{1}{24} } 
  &   \textup{on } y=-\frac{6}{5}x+\frac{149}{5}
\\[5pt]
 m_{\frac{16}{17}} > m_{\frac{9}{23}} < m_{\frac{2}{29}}
  &   \textup{on } y=-\frac{7}{6}x+\frac{215}{6}
\end{array}
\]
Indeed, 
% \[ m_{\frac{7}{11}} = 3\,276\,509 < m_\frac{1}{16} = 3\,524\,578 
% \quad\textup{and}\quad
% m_\frac{9}{11} =16\,964\,653 > m_\frac{3}{16} =16\,609\,837\]
\[
 m_{\frac{13}{14}}=7\,645\,370\,045 > m_{\frac{7}{19}} =6\,684\,339\,842 < m_{\frac{1}{24} } = 7\,778\,742\,049
\]
 and
\[%m_{\frac{8}{17}} = 2\,151\,239\,746 < m_{\frac{1}{23}}=2\,971\,215\,073
 %\quad\textup{and}\quad
m_\frac{16}{17}
 =1\,513\,744\,654\,945 > m_\frac{9}{23}=1\,490\,542\,435\,045<
m_\frac{2}{29}=2\,076\,871\,684\,802\]

\end{remark}
%
%\begin{remark}
% We shall show in a forthcoming paper that For all positive integers $c,d$, we have
% \[ \frac{6c+1}{6c+1+d} <_M \frac{1}{11c+1+d}.\]
%\
%\end{remark}

\section{Proof of Theorem \ref{main thm}} 
Before we start, we recall some facts about Fibonacci numbers and Pell numbers.
Let $\calf_0=0, \calf_1=\calf_2=1, \calf_3=2, \calf_n=\calf_{n-1}+\calf_{n-2}$ denote the Fibonacci sequence, and $\calp_0=0, \calp_1=1, \calp_2=2, \calp_3= 5, \calp_n=2\calp_{n-1}+\calp_{n-2}$ the Pell sequence.
It is well known that the odd indexed Fibonacci and Pell numbers are Markov numbers. In fact $m_\frac{1}{q}=\calf_{2q+1}$ and $m_\frac{n}{n+1}=\calp_{2n+1}$.
\smallskip

This rest of the section is devoted to the proof of the main result. We prove part (b) first. The general strategy is as follows. We want to study the difference $m_{q,p}-m_{q+s,p-t}$ where $E(q,p)$ and $F=(q+s,p-t)$ are two neighboring lattice points on a line of slope $a=-\frac{s}{t}$. Let us define a new pair of points $(E',F')$ by translating the pair $(E,F)$, we write $E'=(q',t+1)$, $F'=(q'+s,1)$. Then $E',F'$ lie on a different line of same slope, and the point $F'$ corresponds to a Fibonacci number.
We compare the difference between the numbers associated to the points $E$ and $F$ with the difference between the numbers associated to the points $E'$ and $F'$ and then let $q'$ go to infinity. The fact that $F'$ corresponds to a Fibonacci number gives us control over the limit.

For part (a), the proof is similar except that we translate that pair $(E,F)$ to a pair $(E',F')$ where now $E'=(q',q'-1)$ corresponds to a Pell number.

We start with the following observation.

\begin{lemma}\label{lemABC}
Assume that $A,B,C$ are non-colinear lattice points such that the triangle $ABC$ does not contain other lattice points besides $A,B,C$. Then
$$|AB|^2+|AC|^2+|BC|^2=3|AB|\,|AC|\,|BC|$$
\end{lemma}
\begin{proof}
The points $A,B,C$ define a triangulation of the torus and the corresponding cluster variables form a cluster in the cluster algebra. Hence the Markov numbers $|AB|,\,|AC|,\,|BC|$ form a Markov triple, see \cite{P, BBH}.
%so the triple of distances $a,b,c$ between them form a Markov triple, and thus satisfy the Markov equation $a^2+b^2+c^2=3abc$; see \cite[\S12]{Springborn}.
%% [Boris Springborn, The hyperbolic geometry of Markov's theorem on Diophantine approximation and quadratic forms, Prop 12.1] which refers to [Penner,56,57]
\end{proof}

\begin{lemma}\label{first few terms}
Let $p,q$ be coprime positive integers and let $f_n=m_{nq,np}$. Thus $f_0=0$ and $f_1=m_{q,p}$ is the Markov number. Then, for $n\ge 2$,
\[   f_n=3f_1 f_{n-1}-f_{n-2}\]

As a consequence, $f_n=\frac{c}{\sqrt{9c^2-4}}(\alpha^n-\alpha^{-n})$ where $c=m_{q,p}$ and $\alpha=(3c+\sqrt{9c^2-4})/2$ is the larger root of $x^2-3cx+1=0$.
\end{lemma}
\begin{proof}
The statement is trivially true for $n=0,1$. 
Assume $n=2$. Let $A=(0,0)$, $B=(q,p)$ and let $C=(q',p') $ be a lattice point such that the triangle $ABC$ does not contain any lattice points besides $A,B$ and $C$.  

Let $E$ be such that $ABEC$ is a parallelogram, thus $E=(q+q',p+p')$,
and let  $D,F$ be such that $ACFD$ is the parallelogram obtained by translation of $ABEC$  by $(-q,-p)$, see Figure  \ref{fig 82}. Note that $f_1=|AB|$ and $f_2=|BD|$.
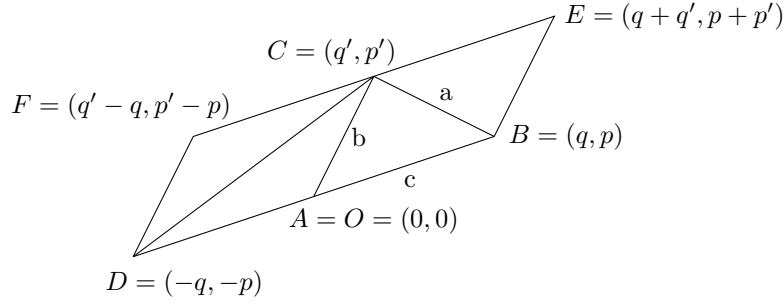
\begin{figure}
\begin{tikzpicture}[scale=.8]
\draw(-3,-1)--(1,2)--(3,1) (1,2)--(0,0);
\draw(-3,-1)--(3,1)--(4,3)--(-2,1)--(-3,-1);
\node at (4.2,1) {$B=(q,p)$};
\node at (6,3) {$E=(q+q',p+p')$};
\node at (1,-0.4) {$A=O=(0,0)$};
\node at (0.3,2.4) {$C=(q',p')$};
\node at (-3.2,1.5) {$F=(q'-q,p'-p)$};
\node at (-2.15,-1.45) {$D=(-q,-p)$};
\node at (2.2,1.7) {a};
\node at (1.6,0.25) {c};
\node at (0.75,1) {b};
\end{tikzpicture}\caption{Picture for Lemma \ref{first few terms}.}\label{fig 82}
\end{figure}

The skein relations imply the following identities
\begin{align} \label{eq51}
|CD| &=\left(|AC|\,|DF| +|AD|\,|CF|\right)/|AF|
=\left(|AC|^2 +|AB|^2 \right)/|BC| 
\end{align}

and
\begin{align}
 |BD|&=\left(|AD|\,|BC|+|AB|\,|CD|\right)/|AC| 
 =|AB|\left(|BC|+|CD|\right)/|AC|
\end{align} 
and using equation (\ref{eq51}) this implies
\[|BD|= |AB|\frac{
|BC|^2+|AC|^2+|AB|^2}{|BC| \,|AC|}
\]
Now Lemma \ref{lemABC} implies 
\[|BD|= |AB|\frac{
3|BC|\,|AC|\,|AB|}{|BC| \,|AC|}
=3|AB|^2
\]
 This proves the case $n=2$. 

For $n>2$, let $A=(0,0), B=(q,p), C=(2q,2p) $ and $D=(nq,np)$. The skein relations imply
\[
|\zg_{AC}^L|
|\zg_{BD}^L|
=
|\zg_{AD}^L|
|\zg_{BC}^L|
+
|\zg_{AB}^L|
|\zg_{CD}^L|
\]
thus
\[|AC|\,|BD| = |AD|\,|BC| + |AB|\,|CD|
\]
and hence
\[f_2f_{n-1} =f_n f_1+f_1f_{n-2} \]
and the result follows since $f_2=3f_1^2$.
\end{proof}
\begin{remark}
An alternative proof of Lemma \ref{first few terms} can be given using bracelets in the torus as follows. If $\zeta$ is a closed curve in the torus then its $k$-bracelet $\Brac_k\zeta$ is defined as the $k$-fold concatenation of $\zeta$ with itself. It is known, see \cite{MSW2}, that the bracelets satisfy the following Chebyshev recursion
\begin{equation}
\label{eqbrac}
\Brac_k\zeta=\zeta\,\Brac_{k-1}\zeta-\Brac_{k-2}\zeta.
\end{equation}
 Now let $\zeta $ be the closed curve obtained from the arc $\ell_{\frac{p}{q}}$ by moving the arc infinitesimally away from the puncture. Then is is shown in \cite{CS5} that $|\zeta|=3\,|\ell_{\frac{p}{q}}|=3f_1$. 
 Similarly, the bracelet $\Brac_n \zeta$ is obtained from the arc $\zg_{nq,np}^L$ by moving the arc infinitesimally away from the puncture. Again we have 
 $|\Brac_n\zeta|=3 f_n$.
 Now equation(\ref{eqbrac}) implies
 $3f_n=3f_1\,3f_{n-1} - 3f_{n-2}$ and we are done.
\end{remark}

\begin{lemma}\label{compare two edges}
%(a) Assume $A,B,C,D\in\mathbb{Z}^2$, $A,B,C$ not colinear, and $D$ is in the angle $\angle C'AB$ (including the boundary) where $C'$ is any point such that the vector $\overrightarrow{AC'}$ is opposite to $\overrightarrow{AC}$ (that is, $\overrightarrow{AD}\in (\mathbb{R}_{\ge0})\overrightarrow{CA}+(\mathbb{R}_{\ge0})\overrightarrow{AB}$). Then $|DB|/|DC|\le |AB|/|AC|$. In particular, if $|AC|\ge |AB|$, then $|DC|\ge |DB|$; if $|AC|>|AB|$, then $|DC|>|DB|$. 

%(b) Assume $O,E,F,G,H,E',F'\in\mathbb{Z}^2$, $O,E,F$ not colinear, $F$ is on the segment $OG$, $OEFH$ and $EFF'E'$ are parallelograms, and $F'$ is inside the angle $\angle GFH$ including the boundary (that is, $\overrightarrow{FF'}\in (\mathbb{R}_{\ge0})\overrightarrow{EO}+(\mathbb{R}_{\ge0})\overrightarrow{OF}$). Then $m_{E'}/m_{F'}\le m_E/m_F$. In particular, if $m_F\ge m_E$, then $m_{F'}\ge m_{E'}$; if $m_F > m_E$, then $m_{F'} > m_{E'}$.
Assume $O,E,F,E',F$ be lattice points such that $O,E,F$ are not colinear, $F'$ is such that $\overrightarrow{FF'}=s \overrightarrow{EO}+t\overrightarrow{OF}$ with $s,t\in \mathbb{R}_{>0}$, $E'$ is such that $\overrightarrow{OE'}=\overrightarrow{OE}+\overrightarrow{FF'}$. Thus $EFF'E'$ is a parallelogram. See Figure \ref{fig 83}. %
Then 
\[|OE|\,|OF'|\ge |OE'|\,|OF|.\]

\end{lemma}
\begin{proof}
Let $O'$ be the lattice point such that $\overrightarrow{OO'}=\overrightarrow{FF'}$. Then the Ptolemy inequality implies 
$|O'E'|\,|OF'|\ge |OE'|\,|O'F'|,$
and the result follows from $|O'E'|=|OE|$ and $|O'F'|=|OF|$.
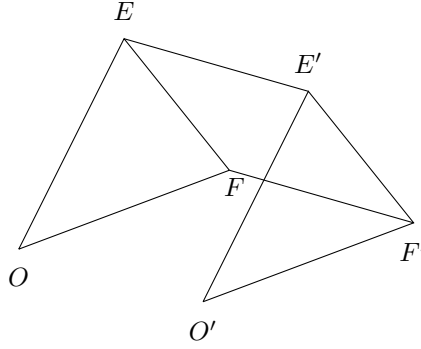
\begin{figure}
\begin{tikzpicture}[scale=.7]
\begin{scope}[shift={(8,0)}]
\draw (2,-1)--(6,0.5)--(4,3)--(2,-1);
\draw (-1.5,0)--(2.5,1.5)--(0.5,4)--(-1.5,0);
\draw (0.5,4)--(4,3);
\draw (2.5,1.5)--(6,.5);
\node [label=below: $O'$] at (2,-1) {};
\node [label=below: $F'$] at (6,.5) {};
\node [label=above: $E'$]at (4,3) {};
\node [label=below: $O$]at (-1.5,0) {};
\node [label=below: $F$] at (2.6,1.7) {};
\node [label=above: $E$]at (0.5,4) {};
\end{scope}
\end{tikzpicture}
\caption{Figure for Lemma \ref{compare two edges} and Lemma \ref{lem:m>m}}\label{fig 83}
\end{figure}
\end{proof}

\begin{lemma}\label{lem:m>m}

Let $q,p,s,t,q'$ be positive integers with $\gcd(s,t)=1$ and $q'>q+t$ and define the lattice points $E=(q,p+s)$, $F=(q+t,p)$,  $E'=(q',s+1)$ and $F'=(q'+t,1)$. Thus $EFF'E'$ is a parallelogram as shown in Figure \ref{fig 83} and the slope of the line segment $\ell_{EF}$ is  $-\frac{s}{t}$.
Assume that $\lim\limits_{q'\to\infty} \frac{|OE'|}{|OF'|}>1.$
%$$\lim_{q'\to\infty} \frac{\ma{s+1}{q'}}{\ma{1}{q'+t}}>1,$$
%Then $\ma{p+s}{q}>\ma{p}{q+t}$. 
Then $|OE|>|OF|$. 
\end{lemma}
%\begin{remark}
% We can restate the condition and the conclusion of the lemma using the $m_{x,y}$ notation as follows. 
% \[
% \lim_{q'\to\infty} \frac{\ma{s+1}{q'}}{\ma{1}{q'+t}}>1 \Rightarrow
%\ma{p+s}{q}>\ma{p}{q+t}.\] 
%\end{remark}
\begin{proof}
For every $q'>q+t$, the points $E,F,F',E'$ satisfy the conditions in Lemma \ref{compare two edges}. Thus
$ |OE|\,|OF'| \ge |OE'|\,|OF|,$
and hence 
\[ \frac{|OE|}{|OF|} \ge \frac{ |OE'|}{|OF'|}.\]
 The result now follows since, by our assumption on the limit, the right hand side is greater than $1$ when $q'$ is large.
\end{proof}

\begin{lemma}\label{lim m/m}

We have $$\lim_{q\to\infty} \frac{\ma{6}{q}}{\ma{1}{q+4}}>1.$$

%(a) For any positive integer $a$, we have 
%$$\lim_{q\to\infty}\frac{m_{q,a}}{m_{q+(a-2),1}}=(\frac{3}{\sqrt{5}})^{a-1}\phi^{-a+3}$$

%(b) The following limits are $>1$:
%$$\lim_{q\to\infty} \frac{\ma{3}{q}}{\ma{1}{q+1}},\ \lim_{q\to\infty} \frac{\ma{4}{q}}{\ma{1}{q+2}},\ \lim_{q\to\infty} \frac{\ma{5}{q}}{\ma{1}{q+3}},\ \lim_{q\to\infty} \frac{\ma{6}{q}}{\ma{1}{q+4}}.$$

%(c) The Markov numbers are decreasing on all lines of slope $-2/1$, $-3/2$, $-4/3$, or $-5/4$.
\end{lemma}

\begin{proof} 
We claim that, for any positive integer $a$, 
\begin{equation}\label{m6n6}
\lim_{q\to\infty} \frac{\ma{a}{q}}{\ma{1}{q+a-2}}=\lim_{n\to\infty} \frac{\ma{a}{an}}{\ma{1}{an+a-2}}
=(\frac{3}{\sqrt{5}\phi})^{a-1}\phi^{2}, \textrm{ where }\phi=\frac{1+\sqrt{5}}{2}= 1.618\cdots
\end{equation}

The first equality follows from Lemma \ref{compare two edges}, which asserts that the sequence of positive numbers $\{\frac{\ma{a}{q}}{\ma{1}{q+a-2}}\}$ is weakly decreasing as $q$ increases. 

Next, we prove the second equality of \eqref{m6n6}. The sequence $\{\ma{1}{q}\}=(1,2,5,13,34,89,\dots)$ is odd-indexed Fibonacci sequence $\{\calf_{2q-1}\}$. So 
$\ma{1}{q}=(\phi^{2q+1} + \phi^{-2q-1})/\sqrt{5}$. 
For any two functions $f(n)$ and $g(n)$ with variable $n$, we write $f\sim g$ if $\lim f(n)/g(n)=1$ as $n\to \infty$. Then $c:=\ma{1}{n}\sim \phi^{2n+1}/\sqrt{5}$, $\alpha=(3c+\sqrt{9c^2-4})/2\sim 3c$, thus
$$\aligned
\ma{a}{an}
&=\frac{c}{\sqrt{9c^2-4}}(\alpha^a-\alpha^{-a})\quad\quad\textrm{(by Lemma \ref{first few terms})}\\
&\sim \frac{c}{3c}\alpha^a\sim \frac{1}{3}(3c)^a\sim 3^{a-1}\phi^{(2n+1)a}/(\sqrt{5})^a\\
\endaligned
$$
On the other hand,
$$\ma{1}{an+a-2}\sim \phi^{2(an+a-2)+1}/\sqrt{5}$$
So
$$\lim_{n\to\infty} \frac{\ma{a}{an}}{\ma{1}{an+a-2}}
=\frac{3^{a-1}\phi^{(2n+1)a}/(\sqrt{5})^a}{ \phi^{2(an+a-2)+1}/\sqrt{5}}
=(\frac{3}{\sqrt{5}})^{a-1}\phi^{-a+3}
=(\frac{3}{\sqrt{5}\phi})^{a-1}\phi^{2}$$
This proves \eqref{m6n6}.

Finally, for $a=6$,  the right hand side of \eqref{m6n6} is $1.026\cdots>1$. 
\end{proof}  

\begin{remark}
Let $a'$ be the solution of the equation $\lim\limits_{q\to\infty} \frac{\ma{a'}{q}}{\ma{1}{q+a'-2}}=1$. Equivalently, by \eqref{m6n6}, $a'$ satisfies the equation
$(\frac{3}{\sqrt{5}\phi})^{a'-1}\phi^{2}=1$, thus $a'=\frac{-2\ln\phi}{\ln(\frac{3}{\sqrt{5}\phi})}+1$.
Let $b$ denote the corresponding slope, thus
$$b=-\frac{a'-1}{a'-2}
=-\frac{\frac{-2\ln\phi}{\ln(\frac{3}{\sqrt{5}\phi})}}{\frac{-2\ln\phi}{\ln(\frac{3}{\sqrt{5}\phi})}-1}
=\frac{-2}{2+ \frac{\ln(\frac{3}{\sqrt{5}\phi})}{\ln\phi} }
=\frac{-2}{2+\frac{\ln(3/\sqrt{5})-\ln\phi}{\ln\phi}}
=\frac{-2}{1+\frac{\ln(3/\sqrt{5})}{\ln\phi}}$$
Or we can express $b$ as
$$b=\frac{-2\ln\phi}{\ln\phi+\ln(3/\sqrt{5})}
=-\frac{\ln((\frac{1+\sqrt{5}}{2})^2)}{\ln(\frac{1+\sqrt{5}}{2}\cdot\frac{3}{\sqrt{5}})}
=-\frac{\ln\frac{3+\sqrt{5}}{2}}{\ln \frac{3(1+\sqrt{5})}{2\sqrt{5}}}
=-1.241668489\cdots$$ 
We conjecture that $b$ is best replacement of $-5/4$ in Theorem \ref{main thm} (b): 
\end{remark}
\begin{conjecture} \label{decreasing slope}
The Markov distances decrease with $x$ along any line with slope $\le -\frac{\ln\frac{3+\sqrt{5}}{2}}{\ln \frac{3(1+\sqrt{5})}{2\sqrt{5}}}$.
\end{conjecture}

\subsection{Proof of Theorem \ref{main thm} (b)}

%\begin{proof}[Proof of Theorem \ref{main thm} (b)]
By Lemma \ref{lim m/m}(c), the Markov numbers are decreasing on any line with slope $-5/4$. So we are left to show the decreasing property on lines of slope $<-5/4$.

Let $(q_1,p_1), (q_2,p_2)$ be two lattice points satisfying $q_i>p_i>0$ (for $i=1,2$), $q_2>q_1$, $(p_2-p_1)/(q_2-p_1) < -\frac{5}{4}$. We claim that $\ma{p_1}{q_1}>\ma{p_2}{q_2}$. 

Consider the line $L$ through the point $(q_2,p_2)$ with slope $ -\frac{5}{4}$. If there exists an integer point on $L$, denoted $(q_3,p_3)$, such that $q_3\le q_1$ and $p_3\le p_1$. In other words, the point $(q_3,p_3)$ lies (weakly) southwest of the point  $(q_1,p_1)$. Then we have $\ma{p_1}{q_1}>\ma{p_1}{q_3}>\ma{p_3}{q_3}>\ma{p_2}{q_2}$
 where the first inequality follows from the constant numerator theorem \ref{constant numerator}, the second from the constant denominator theorem \ref{constant denominator},  
while the third inequality uses the decreasing property when the slope is equal to $-5/4$, see the left picture in Figure \ref{fig 62_1}. This completes the proof in this case.
 
Now suppose that there is no integer point on the line $L$ that lies southwest of $(q_1,p_1)$. Then there are two consecutive integer points $(q_3,p_3)$ and  $(q_4,p_4)=(q_3-4,p_3+5)$ such that $q_4<q_1<q_3$ and $p_4>p_1>p_3$,  see the right picture in Figure \ref{fig 62_1}. Then  there are only 6 possibilities for $(q_1,p_1)$ 
$$(q_3+1,p_3-1), (q_3+2,p_3-1), (q_3+3,p_3-1), (q_3+2,p_3-2), (q_3+3,p_3-2), (q_3+3,p_3-3).$$
 It follows from Theorem \ref{thm 1new} that the smallest $m_{q,p}$ for the above 6 pairs $(q,p)$ is obtained when $(q,p)=(q_3+1,p_3-1)$.
%  So it suffices to show that $m_{q_3+1,p_3-1}>m_{q_4,p_4}(=m_{q_3+4,p_3-5})$. But the line through $(q_3+1,p_3-1)$ and $(q_3+4,p_3-5)$ has slope $-4/3$. So the inequality follows from Lemma \ref{lim m/m}. 
Thus we have $m_{q_1,p_1}\ge m_{q_3+1,p_3-1}>m_{q_4,p_4}\ge m_{q_2,p_2}$ where the second inequality follows from Theorem \ref{constant sum} and the last inequality from Lemma \ref{lim m/m}(c) since the slope of the line $L$ is $-5/4$.
This completes the proof of part (b) of Theorem \ref{main thm}.%\end{proof}
\begin{figure}[h]
\begin{tikzpicture}[scale=.35]
\node (v1)[label=right: ${(q_1,p_1)}$] at (6,13) {};
\node (v2)[label=left: ${(q_2,p_2)}$] at (10,1) {};
\node (v3)[label=left: ${(q_3,p_3)}$] at (2,11) {};
\draw  (12,-1.5)--(0,13.5);
\node at (12,0){$L$};
\fill (v1) circle[radius=4pt];
\fill (v2) circle[radius=4pt];
\fill (v3) circle[radius=4pt];
\begin{scope}[shift={(18,0)}]

\foreach \x in {2,...,6}
{
\draw (\x,11) -- (\x,6);
}
\foreach \y in {6,...,11}
{
\draw (2,\y) -- (6,\y);
}
%\node (v1)[label=above: ${(q_1,p_1)}$] at (3,10) {};
\node (v2)[label=left: ${(q_2,p_2)}$] at (10,1) {};
\node (v3)[label=left: ${(q_4,p_4)}$] at (2,11) {};
\node (v4)[label=right: ${(q_3,p_3)}$] at (6,6) {};
\draw  (12,-1.5)--(0,13.5);
\node at (12,0){$L$};
\fill (v1) circle[radius=4pt];
\fill (v2) circle[radius=4pt];
\fill (v3) circle[radius=4pt];
\fill (v4) circle[radius=4pt];
\draw(3,10) circle[radius=4pt];
\draw(4,10) circle[radius=4pt];
\draw(5,10) circle[radius=4pt];
\draw(4,9) circle[radius=4pt];
\draw(5,9) circle[radius=4pt];
\draw(5,8) circle[radius=4pt];
\end{scope}
\end{tikzpicture}
\caption{Proof of Conjecture 6.2 (b). Slope of $L$ is $-5/4$. Left: the case $q_3\le q_1$ and $p_3\le p_1$; Right: the other case, where $(q_1,p_1)$ is one of the 6 circled points.}\label{fig 62_1}
\end{figure}
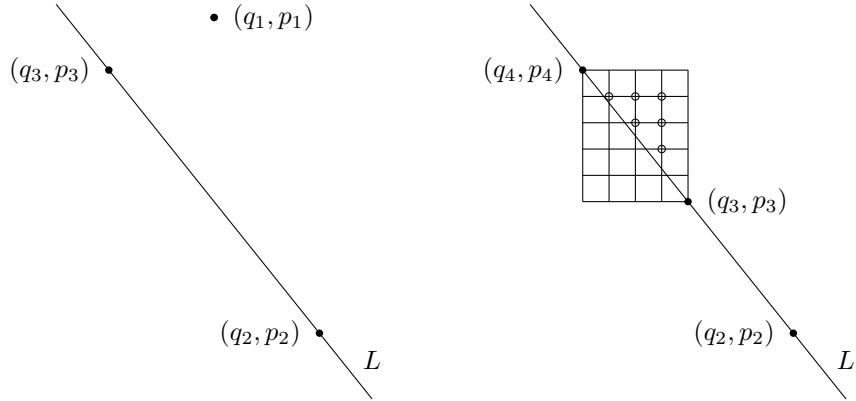

\subsection{Proof of Theorem \ref{main thm} (a) }

The proof of part (a) of the theorem is similar to the proof of part (b) except that instead of using Fibonacci numbers $m_{q,1}$ we will use  Pell numbers $m_{q,q-1}$.

We first show an analogue of Lemma \ref{lim m/m}:

\begin{lemma}\label{lim m/m 2}
We have
$$\lim_{q\to\infty} \frac{\ma{q-9}{q+7}}{\ma{q-1}{q}}>1.$$
%As a consequence, the Markov numbers increase with $x$ along any line of slope $-8/7$.
\end{lemma}

\begin{proof}
Similar to the proof of Lemma \ref{lim m/m}, 
we claim that, for any positive integer $a$, 
\begin{equation}\label{man}
\lim_{q\to\infty} \frac{\ma{q-a-1}{q+a-1}}{\ma{q-1}{q}}=\lim_{n\to\infty} \frac{\ma{2an-2a}{2an}}{\ma{2an-a}{2an-a+1}}
=\frac{3^{2a-1}}{2^{3a-1}}(2-\sqrt{2}). 
\end{equation}

The first equality follows from Lemma \ref{compare two edges}, which asserts that the sequence of positive numbers  $\{\frac{\ma{q-a-1}{q+a-1}}{\ma{q-1}{q}}\}$ is weakly decreasing as $q$ increases. 

Next, we prove the second equality of \eqref{man}. 
The Markov number $\ma{q-1}{q}$ is the Pell number $\calp_{2q-1}$. The sequence $\{\ma{q-1}{q}\}=(1,5,29,169,\dots)$ (for $q\ge 1$) is defined by the recursive relation $a_q=6a_{q-1}-a_{q-2}$ with initial conditions $a_1=1$, $a_2=5$. 
Using the well-known formula 
$\calp_n=\frac{(1+\sqrt{2})^n-(1-\sqrt{2})^n}{2\sqrt{2}}$ we get 
\[\ma{q-1}{q}=\calp_{2q-1}=\frac{(1+\sqrt{2})^{2q-1}-(1-\sqrt{2})^{2q-1}}{2\sqrt{2}}=\frac{2-\sqrt{2}}{4}\psi^{q} + \frac{2+\sqrt{2}}{4}\psi^{-q},\]
 where $\psi=3+\sqrt{8}.$
%$$\ma{q-1}{q}=\frac{2-\sqrt{2}}{4}\psi^{q} + \frac{2+\sqrt{2}}{4}\psi^{-q}\textrm{ where }\psi=3+\sqrt{8}.$$
By Lemma \ref{first few terms} (with $c=\ma{n-1}{n}$,  $\alpha=(3c+\sqrt{9c^2-4})/2\sim 3c$):
$$\ma{2a(n-1)}{2an}=\frac{c}{\sqrt{9c^2-4}}(\alpha^{2a}-\alpha^{-2a})
\sim\frac{1}{3}(3c)^{2a}
\sim3^{2a-1}(\ma{n-1}{n})^{2a}
\sim3^{2a-1}\left(\textstyle\frac{2-\sqrt{2}}{4}\right)^{\!2a}\psi^{2an}.
$$
On the other hand
$$\ma{2an-a}{2an-a+1}\sim \left(\textstyle\frac{2-\sqrt{2}}{4}\right)\psi^{2an-a+1}.$$
So
$$\lim_{n\to\infty} \frac{\ma{2an-2a}{2an}}{\ma{2an-a}{2an-a+1}}
=\frac{3^{2a-1}\left(\textstyle\frac{2-\sqrt{2}}{4}\right)^{\!2a}\psi^{2an}}{  \left(\textstyle\frac{2-\sqrt{2}}{4}\right)\psi^{2an-a+1} }
=\Big(\frac{3(2-\sqrt{2})}{4}\Big)^{2a-1}\psi^{a-1}=\frac{3^{2a-1}}{2^{3a-1}}(2-\sqrt{2})
$$
This proves \eqref{man}.

Finally, for $a=8$,  the right hand side of \eqref{man} is $1.00200118\cdots>1$. 
\end{proof}  
\begin{remark}
Let $a''$ be the solution of the equation $\lim_{q\to\infty} \frac{\ma{q-a''-1}{q+a''-1}}{\ma{q-1}{q}}=1$. Equivalently, by \eqref{man}, $a''$ satisfies the equation
$\frac{3^{2a''-1}}{2^{3a''-1}}(2-\sqrt{2})=1$, thus
$(\frac{9}{8})^{a''}=\frac{3}{2(2-\sqrt{2})}$, $a''=\frac{\ln\frac{3}{2(2-\sqrt{2})}}{\ln(9/8)}$.
Let $b$ be the corresponding slope, that is,
$$b=-\frac{a''}{a''-1}
=-\frac{\frac{\ln\frac{3}{2(2-\sqrt{2})}}{\ln(9/8)}}{\frac{\ln\frac{3}{2(2-\sqrt{2})}}{\ln(9/8)}-1}
=-\frac{{\ln\frac{3}{2(2-\sqrt{2})}}}{{\ln\frac{3}{2(2-\sqrt{2})}}-{\ln\frac{9}{8}}}
=-\frac{\ln\frac{3(2+\sqrt{2})}{4}}{\ln\frac{2(2+\sqrt{2})}{3}}= -1.14320438\cdots$$                                                             
We conjecture that $b$ is best replacement of $-8/7$ in Theorem \ref{main thm} (a): 
\end{remark} 

\begin{conjecture}\label{increasing slope}
The Markov distances increase with $x$ along any line with slope $\ge -\frac{\ln\frac{3(2+\sqrt{2})}{4}}{\ln\frac{2(2+\sqrt{2})}{3}}$.
\end{conjecture}

For a line whose slope is not within the range given in Conjectures \ref{decreasing slope} and \ref{increasing slope}, we conjecture that the Markov numbers form a strictly anti-unimodal sequence provided there are enough lattice points on the line. More precisely, we propose the following. 

\begin{conjecture}\label{decreasing-increasing slope}
Given a line $y=ax+b$ whose slope $a$ is rational and satisfies $-\frac{\ln\frac{3+\sqrt{5}}{2}}{\ln \frac{3(1+\sqrt{5})}{2\sqrt{5}}}< a < -\frac{\ln\frac{3(2+\sqrt{2})}{4}}{\ln\frac{2(2+\sqrt{2})}{3}}$, we let $(q_1,p_1),\ldots,(q_n,p_n) $ be the  lattice points on the line that satisfy $1\le p_i\le q_i$, and we arrange them such that $q_1<q_2< \cdots <q_n$.  Then there exists $b_0>0$ such that for all lines $y=ax+b$ with $b\ge b_0$ and contain at least a lattice points, the sequence of Markov distances $m_{q_i,p_i}$ is strictly anti-unimodal, that is,
\[ m_{q_1,p_1} >  m_{q_2,p_2} > \cdots >  m_{q_j,p_j} < m_{q_{j+1},p_{j+1}} < \cdots < m_{q_n,p_n}, \text{ for some }\ 2\le j\le n-1.\]
\end{conjecture}

\begin{remark}
In this conjecture we do not require that the coordinates $(q_i,p_i)$ are relatively prime. 
\end{remark}

\medskip
We now return to the proof of our main result. The next lemma proves Theorem \ref{main thm} (a) in the case where the slope is equal to $-\frac{7}{8}$.
\begin{lemma}\label{lem 78}  Let $p<q$ be positive integers. Then $\ma{p}{q}<\ma{p-8}{q+7}$.
In particular, the Markov numbers increase with $x$ on any line of slope $a=-\frac{7}{8}$. \end{lemma}
\begin{proof}
 Consider the points $E=(q+7,p-8), F=(q,p), E'=(q'+7,q'-9), F'=(q',q'-1)$. We verify that, when $q'$ is sufficiently large, the condition 
 $\overrightarrow{FF'}= s\overrightarrow{EO}+t\overrightarrow{OF}$ with $s,t\in \mathbb{R}_{>0}$ in Lemma \ref{compare two edges} holds. Indeed, 
% let $v_1,v_2\in\mathbb{R}$ be the coefficients in
%$\overrightarrow{FF'}=v_1\overrightarrow{EO}+v_2\overrightarrow{OF}$. Then,
$$\begin{bmatrix}s\\t\end{bmatrix}
=
\begin{bmatrix}-q-7&q\\-p+8&p\end{bmatrix}^{-1}\begin{bmatrix}q'-q\\q'-p-1\end{bmatrix}
=
\frac{1}{7p+8q}\begin{bmatrix}(q-p)q'-q\\ (q-p+15)q'-7p-8q-7\end{bmatrix}
$$
which is in $\mathbb{R}_{>0}^2$ as $q'\to\infty$ because $q>p$.
Thus we can apply Lemma \ref{compare two edges} to get an inequality
\[|OE|\,|OF'|\ge |OE'|\,|OF| \qquad
\textup{ and thus }  \qquad \frac{\ma{p-8}{q+7}}{\ma{p}{q}}\ge \frac{\ma{q'-9}{q'+7}}{\ma{q'-1}{q'}}.\]
%$$\frac{\ma{q'-9}{q'+7}}{\ma{q'-1}{q'}}=m_{E'}/m_{F'}\le m_{E}/m_F=\frac{\ma{p-8}{q+7}}{\ma{p}{q}}$$
%
The limit as $q'\to \infty$  of the right hand side in the last inequality  is greater that 1, by Lemma \ref{lim m/m 2}. Thus the left hand side is also greater than 1. 
\end{proof}

\begin{proof}[Proof of Theorem \ref{main thm}(a)]
 Let $(q_1,p_1), (q_2,p_2)$ be two lattice points satisfying $0<p_1<q_1$, $0<p_2<q_2$
 such that $q_1<q_2$ and the slope of the line segment from $(q_1,p_1)$ to $ (q_2,p_2)$ is greater or equal to $ -\frac{8}{7}$, that is  $(p_2-p_1)/(q_2-p_1) \ge -\frac{8}{7}$. We need to show that $\ma{p_1}{q_1}<\ma{p_2}{q_2}$. Since we already know that the conclusion is true if $p_2\ge p_1$, we assume $p_2<p_1$ in the rest of the proof. 
 We also may assume that the slope is strictly smaller that $ -\frac{8}{7}$ because of Lemma \ref{lem 78}.

Consider the line $L$ through the point $(q_1,p_1)$ with slope $ -\frac{8}{7}$. If there exists a lattice point $(q_3,p_3)$  on $L$  that lies  (weakly) southwest of  $(q_2,p_2)$, then  Theorem \ref{thm 1new} implies  $\ma{p_1}{q_1}<\ma{p_3}{q_3}<\ma{p_2}{q_2}$ and we are done. See the left picture in Figure \ref{fig 62_2}.

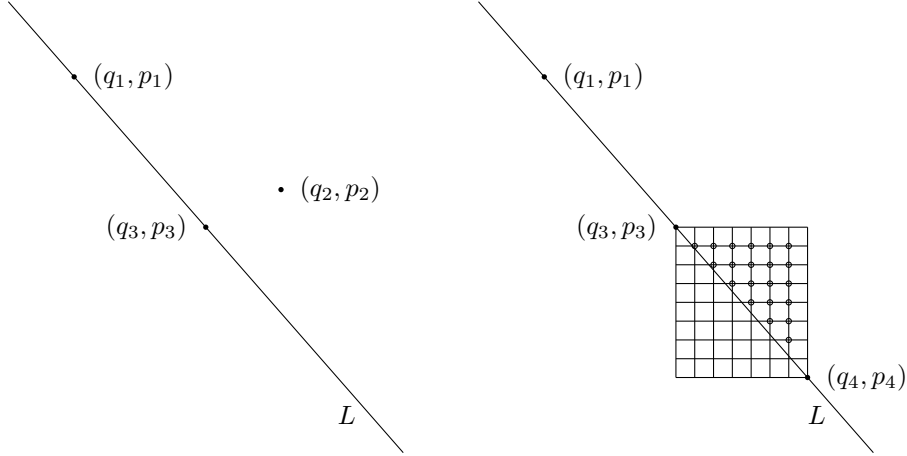
\begin{figure}[h]
\begin{tikzpicture}[scale=.25]
\node (v1)[label=right: ${(q_1,p_1)}$] at (-5,19) {};
\node (v2)[label=right: ${(q_2,p_2)}$] at (6,13) {};
\node (v3)[label=left: ${(q_3,p_3)}$] at (2,11) {};
\draw  (-8.5,23)--(12.5,-1);
\node at (9.5,1){$L$};
\fill (v1) circle[radius=4pt];
\fill (v2) circle[radius=4pt];
\fill (v3) circle[radius=4pt];

\begin{scope}[shift={(25,0)}]

\foreach \x in {2,...,9}
{
\draw (\x,11) -- (\x,3);
}
\foreach \y in {3,...,11}
{
\draw (2,\y) -- (9,\y);
}
\node (v1)[label=right: ${(q_1,p_1)}$] at (-5,19) {};
\node (v3)[label=left: ${(q_3,p_3)}$] at (2,11) {};
\node (v4)[label=right: ${(q_4,p_4)}$] at (9,3) {};
\draw  (-8.5,23)--(12.5,-1);
\node at (9.5,1){$L$};
\fill (v1) circle[radius=4pt];
\fill (v3) circle[radius=4pt];
\fill (v4) circle[radius=4pt];
\foreach \x in {3,...,8}
{
  \draw(\x,10) circle[radius=4pt];
}
\foreach \x in {4,...,8}
{
  \draw(\x,9) circle[radius=4pt];
}
\foreach \x in {5,...,8}
{
  \draw(\x,8) circle[radius=4pt];
}
\foreach \x in {6,...,8}
{
  \draw(\x,7) circle[radius=4pt];
}
\foreach \x in {7,...,8}
{
  \draw(\x,6) circle[radius=4pt];
}
\foreach \x in {8,...,8}
{
  \draw(\x,5) circle[radius=4pt];
}
\end{scope}
\end{tikzpicture}

\caption{Proof of  Theorem \ref{main thm} (a). Slope of $L$ is $-8/7$. Left: the case $q_3\le q_2$ and $p_3\le p_2$; Right: the other case, where $(q_2,p_2)$ is one of the 21 circled points.}\label{fig 62_2}
\end{figure}

The only unproved situation is when there are two consecutive lattice points $(q_3,p_3)$ and  $(q_4,p_4)=(q_3+7,p_3-8)$ such that $q_3<q_2<q_4$ and $p_3>p_2>p_4$. Then there are precisely 21 possibilities  for $(q_2,p_2)$, namely the lattice points in the interior of the right triangle with vertices $(q_3,p_3),(q_4,p_3), (q_4,p_4)$; see the right picture in Figure \ref{fig 62_2}. Observe that any line that passes through $(q_3,p_3)$ and one of these 21 points has slope $\ge -1$. Thus $\ma{p_2}{q_2}>\ma{p_3}{q_3}$. On the other hand, we have $\ma{p_1}{q_1}<\ma{p_3}{q_3}$ from Lemma \ref{lem 78} and thus we conclude that $\ma{p_2}{q_2}>\ma{p_1}{q_1}$.
\end{proof}

\section{A comment on the Markov distance}\label{sect 5}
In this section we show that a modified version of the Markov distance satisfies the triangle inequality.
\begin{corollary}
Let $A,B,C$ be distinct lattice points. Then $$3\, |AB|\cdot|BC|\ge|AC|$$
and the equality holds if and only if $\overrightarrow{AB}=\overrightarrow{BC}=\left(
\begin{smallmatrix}
p\\q 
\end{smallmatrix}
\right)$ where $p,q$ are relatively prime.

As a consequence, if we define $d:\mathbb{Z}^2\times\mathbb{Z}^2\to \mathbb{R}_{\ge0}$ as 
$$d(A,B)=\begin{cases}
&\ln(3|AB|), \textrm{ if $A\neq B$};\\
&0, \textrm{ if $A=B$}.
\end{cases}$$
then $d(A,B)+d(B,C)\ge d(A,C)$ for any $A,B,C\in\mathbb{Z}^2$. Moreover, the equality holds if $A=B$, or $B=C$, or $\overrightarrow{AB}=\overrightarrow{BC}=(p,q)$ where $p,q$ are relatively prime.
\end{corollary}
\begin{proof}
First we prove the inequality for the special case when $B$ is the midpoint of the segment $\ell_{AC}$ (that is, $\overrightarrow{AB}=\overrightarrow{BC}$). For this, we shall use Lemma \ref{first few terms} and the notations therein.
Let $\frac{p}{q}$ be the slope of the line segment $\ell_{AB}$. Then $\ell_{AB}$ runs through $n+1$ lattice points $A,A+(q,p),,A+2(q,p),\ldots, A+n(q,p)=B$.
Then we want to show
$$3|AB|^2\ge |AC|$$
or equivalently, $3f_n^2\ge f_{2n}$ where $f_n=\frac{c}{\sqrt{9c^2-4}}(\alpha^n-\alpha^{-n})$ where $\alpha=(3c+\sqrt{9c^2-4})/2>1$ is the larger root of $x^2-3cx+1=0$ (where $c\in\mathbb{Z}_{>0}$).  We have
\[\begin{array}{rcl}
3f_n^2-f_{2n}&=& 3\left( \frac{c}{\sqrt{9c^2-4}} (\za^n-\za^{-n})\right)^2-\frac{c}{\sqrt{9c^2-4}}(\za^{2n}-\za^{-2n})\\
&=& \frac{c}{\sqrt{9c^2-4}} (\za^n-\za^{-n}) 
\left(3 \frac{c}{\sqrt{9c^2-4}} (\za^n-\za^{-n})-(\za^n-\za^{-n})\right)\\
&=& \frac{c}{\sqrt{9c^2-4}} (\za^n-\za^{-n}) 
\left((3 \frac{c}{\sqrt{9c^2-4}} -1)(\za^n-\za^{-n})\right)\\
\end{array}\]
and each term in this product is positive, since $ \frac{c}{\sqrt{9c^2-4}} >1, \za>1$ and $n>1$. Thus $3f_n^2\ge f_{2n}$.
%$$\aligned
%3f_n^2\ge f_{2n}
%&\Longleftrightarrow
%3\big(\frac{c}{\sqrt{9c^2-4}}(\alpha^n-\alpha^{-n})\big)^2\ge \frac{c}{\sqrt{9c^2-4}}(\alpha^{2n}-\alpha^{-2n})\\
%&\Longleftrightarrow \frac{3c}{\sqrt{9c^2-4}}(\alpha^n-\alpha^{-n})\ge \alpha^n+\alpha^{-n}\\
%&\Longleftrightarrow \big(\frac{3c}{\sqrt{9c^2-4}}-1\big)\alpha^n\ge \big(\frac{3c}{\sqrt{9c^2-4}}+1\big)\alpha^{-n}\\
%&\Longleftrightarrow \alpha^{2n}\ge \frac{\frac{3c}{\sqrt{9c^2-4}}+1}{\frac{3c}{\sqrt{9c^2-4}}-1}
%=\frac{\frac{3c+\sqrt{9c^2-4}}{2}}{\frac{3c-\sqrt{9c^2-4}}{2}}=\frac{\alpha}{\alpha^{-1}}=\alpha^2
%\Longleftrightarrow \alpha^{n}\ge \alpha
%\\
%\endaligned$$
%where the last inequality is true because $\alpha>1$ and $n>1$; moreover, the equality holds only when $n=1$.  

Now for $A,B,C$ in general, consider the parallelogram $ACDE$ with center $B$. 
\begin{figure}[h]
\begin{tikzpicture}[scale=.7]
\draw (0,0)--(3,1)--(2,1.5)--(0,0);
\draw[dashed] (0,0)--(1,2)--(4,3)--(3,1) (1,2)--(2,1.5)--(4,3);
\node [label=left: $A$] at (0,0) {};
\node [label=below: $B$] at (2,1.7) {};
\node [label=below: $C$]at (3,1) {};
\node [label=right: $D$]at (4,3) {};
\node [label=above: $E$]at (1,2) {};
\end{tikzpicture}
\end{figure}

 Then
by the first part of the proof,  $3|AB|^2\ge |AD|$ and $3|BC|^2\ge |BE|$.
By Corollary \ref{Ptolemy inequality}, $|AD|\cdot |CE|\ge |AC|^2+|AE|^2\ge |AC|^2$. 
So we have $(3|AB|^2)(3|BC|^2)\ge|AC|^2$, taking square roots on both sides gives $3|AB|\cdot|BC|\ge|AC|$.

The consequence follows immediately.
\end{proof}

At the end of the paper, we propose a conjecture that implies the uniqueness conjecture of Markov numbers. 

We call a real number $r$ a Markov distance if there exist two lattice points $A,B$ such that $r$ is the Markov distance $|AB|$ between $A$ and $B$. Note that a Markov distance is always a nonnegative integer.
\begin{conjecture}
[Uniqueness conjecture of Markov distances]
Every Markov distance is equal to  $m_{q,p}$ for a unique the pair of integers $(p,q)$ satisfying $0\le p\le q$. 
\end{conjecture}

\medskip

\noindent {Acknowledgement}. {The authors wish to thank the referee for providing many valuable suggestions. }

{}
\end{document}

%% file: figschema.pdf_tex
%% Creator: Inkscape 1.0 (4035a4f, 2020-05-01), www.inkscape.org
%% PDF/EPS/PS + LaTeX output extension by Johan Engelen, 2010
%% Accompanies image file 'figschema.pdf' (pdf, eps, ps)
%%
%% To include the image in your LaTeX document, write
%%   \input{<filename>.pdf_tex}
%%  instead of
%%   \includegraphics{<filename>.pdf}
%% To scale the image, write
%%   \def\svgwidth{<desired width>}
%%   \input{<filename>.pdf_tex}
%%  instead of
%%   \includegraphics[width=<desired width>]{<filename>.pdf}
%%
%% Images with a different path to the parent latex file can
%% be accessed with the `import' package (which may need to be
%% installed) using
%%   \usepackage{import}
%% in the preamble, and then including the image with
%%   \import{<path to file>}{<filename>.pdf_tex}
%% Alternatively, one can specify
%%   \graphicspath{{<path to file>/}}
%% 
%% For more information, please see info/svg-inkscape on CTAN:
%%   http://tug.ctan.org/tex-archive/info/svg-inkscape
%%
\begingroup%
  \makeatletter%
  \providecommand\color[2][]{%
    \errmessage{(Inkscape) Color is used for the text in Inkscape, but the package 'color.sty' is not loaded}%
    \renewcommand\color[2][]{}%
  }%
  \providecommand\transparent[1]{%
    \errmessage{(Inkscape) Transparency is used (non-zero) for the text in Inkscape, but the package 'transparent.sty' is not loaded}%
    \renewcommand\transparent[1]{}%
  }%
  \providecommand\rotatebox[2]{#2}%
  \newcommand*\fsize{\dimexpr\f@size pt\relax}%
  \newcommand*\lineheight[1]{\fontsize{\fsize}{#1\fsize}\selectfont}%
  \ifx\svgwidth\undefined%
    \setlength{\unitlength}{364.6865503bp}%
    \ifx\svgscale\undefined%
      \relax%
    \else%
      \setlength{\unitlength}{\unitlength * \real{\svgscale}}%
    \fi%
  \else%
    \setlength{\unitlength}{\svgwidth}%
  \fi%
  \global\let\svgwidth\undefined%
  \global\let\svgscale\undefined%
  \makeatother%
  \begin{picture}(1,1.13611422)%
    \lineheight{1}%
    \setlength\tabcolsep{0pt}%
    \put(0,0){\includegraphics[width=\unitlength,page=1]{figschema.pdf}}%
    \put(0.60265655,0.4133866){\makebox(0,0)[lt]{\lineheight{1.25}\smash{\begin{tabular}[t]{l}$(q,p)$\end{tabular}}}}%
    \put(0,0){\includegraphics[width=\unitlength,page=2]{figschema.pdf}}%
    \put(0.43813137,0.61904251){\makebox(0,0)[lt]{\lineheight{1.25}\smash{\begin{tabular}[t]{l}$(q-4,p+5)$\end{tabular}}}}%
    \put(0.01859307,0.72598358){\makebox(0,0)[lt]{\lineheight{1.25}\smash{\begin{tabular}[t]{l}$(q-7,p+8)$\end{tabular}}}}%
    \put(0.12553433,0.33112422){\makebox(0,0)[lt]{\lineheight{1.25}\smash{\begin{tabular}[t]{l}smaller than $m_{\frac{p}{q}}$\end{tabular}}}}%
    \put(0.557412,0.78356752){\makebox(0,0)[lt]{\lineheight{1.25}\smash{\begin{tabular}[t]{l}larger than $m_{\frac{p}{q}}$\end{tabular}}}}%
    \put(0,0){\includegraphics[width=\unitlength,page=3]{figschema.pdf}}%
  \end{picture}%
\endgroup%

%% file: figure11.pdf_tex
%% Creator: Inkscape 1.0 (4035a4f, 2020-05-01), www.inkscape.org
%% PDF/EPS/PS + LaTeX output extension by Johan Engelen, 2010
%% Accompanies image file 'figure11.pdf' (pdf, eps, ps)
%%
%% To include the image in your LaTeX document, write
%%   \input{<filename>.pdf_tex}
%%  instead of
%%   \includegraphics{<filename>.pdf}
%% To scale the image, write
%%   \def\svgwidth{<desired width>}
%%   \input{<filename>.pdf_tex}
%%  instead of
%%   \includegraphics[width=<desired width>]{<filename>.pdf}
%%
%% Images with a different path to the parent latex file can
%% be accessed with the `import' package (which may need to be
%% installed) using
%%   \usepackage{import}
%% in the preamble, and then including the image with
%%   \import{<path to file>}{<filename>.pdf_tex}
%% Alternatively, one can specify
%%   \graphicspath{{<path to file>/}}
%% 
%% For more information, please see info/svg-inkscape on CTAN:
%%   http://tug.ctan.org/tex-archive/info/svg-inkscape
%%
\begingroup%
  \makeatletter%
  \providecommand\color[2][]{%
    \errmessage{(Inkscape) Color is used for the text in Inkscape, but the package 'color.sty' is not loaded}%
    \renewcommand\color[2][]{}%
  }%
  \providecommand\transparent[1]{%
    \errmessage{(Inkscape) Transparency is used (non-zero) for the text in Inkscape, but the package 'transparent.sty' is not loaded}%
    \renewcommand\transparent[1]{}%
  }%
  \providecommand\rotatebox[2]{#2}%
  \newcommand*\fsize{\dimexpr\f@size pt\relax}%
  \newcommand*\lineheight[1]{\fontsize{\fsize}{#1\fsize}\selectfont}%
  \ifx\svgwidth\undefined%
    \setlength{\unitlength}{271.97610338bp}%
    \ifx\svgscale\undefined%
      \relax%
    \else%
      \setlength{\unitlength}{\unitlength * \real{\svgscale}}%
    \fi%
  \else%
    \setlength{\unitlength}{\svgwidth}%
  \fi%
  \global\let\svgwidth\undefined%
  \global\let\svgscale\undefined%
  \makeatother%
  \begin{picture}(1,0.44592273)%
    \lineheight{1}%
    \setlength\tabcolsep{0pt}%
    \put(0,0){\includegraphics[width=\unitlength,page=1]{figure11.pdf}}%
    \put(0.13301481,0.03039769){\color[rgb]{0,0,0}\makebox(0,0)[lt]{\lineheight{0}\smash{\begin{tabular}[t]{l}1\end{tabular}}}}%
    \put(0.29847053,0.03039769){\color[rgb]{0,0,0}\makebox(0,0)[lt]{\lineheight{0}\smash{\begin{tabular}[t]{l}1\end{tabular}}}}%
    \put(0.13301481,0.22894456){\color[rgb]{0,0,0}\makebox(0,0)[lt]{\lineheight{0}\smash{\begin{tabular}[t]{l}1\end{tabular}}}}%
    \put(0.29847053,0.22894456){\color[rgb]{0,0,0}\makebox(0,0)[lt]{\lineheight{0}\smash{\begin{tabular}[t]{l}1\end{tabular}}}}%
    \put(0.13301481,0.39991547){\color[rgb]{0,0,0}\makebox(0,0)[lt]{\lineheight{0}\smash{\begin{tabular}[t]{l}1\end{tabular}}}}%
    \put(0.29847053,0.39991547){\color[rgb]{0,0,0}\makebox(0,0)[lt]{\lineheight{0}\smash{\begin{tabular}[t]{l}1\end{tabular}}}}%
    \put(0.03453459,0.13502397){\color[rgb]{0,0,0}\makebox(0,0)[lt]{\lineheight{0}\smash{\begin{tabular}[t]{l}2\end{tabular}}}}%
    \put(0.1999903,0.13502397){\color[rgb]{0,0,0}\makebox(0,0)[lt]{\lineheight{0}\smash{\begin{tabular}[t]{l}2\end{tabular}}}}%
    \put(0.39302197,0.13502397){\color[rgb]{0,0,0}\makebox(0,0)[lt]{\lineheight{0}\smash{\begin{tabular}[t]{l}2\end{tabular}}}}%
    \put(0.03453459,0.30047968){\color[rgb]{0,0,0}\makebox(0,0)[lt]{\lineheight{0}\smash{\begin{tabular}[t]{l}2\end{tabular}}}}%
    \put(0.1999903,0.30047968){\color[rgb]{0,0,0}\makebox(0,0)[lt]{\lineheight{0}\smash{\begin{tabular}[t]{l}2\end{tabular}}}}%
    \put(0.39302197,0.30047968){\color[rgb]{0,0,0}\makebox(0,0)[lt]{\lineheight{0}\smash{\begin{tabular}[t]{l}2\end{tabular}}}}%
    \put(0.10961585,0.28983505){\color[rgb]{0,0,0}\makebox(0,0)[lt]{\lineheight{0}\smash{\begin{tabular}[t]{l}3\end{tabular}}}}%
    \put(0.27507157,0.28983505){\color[rgb]{0,0,0}\makebox(0,0)[lt]{\lineheight{0}\smash{\begin{tabular}[t]{l}3\end{tabular}}}}%
    \put(0.10961585,0.12437936){\color[rgb]{0,0,0}\makebox(0,0)[lt]{\lineheight{0}\smash{\begin{tabular}[t]{l}3\end{tabular}}}}%
    \put(0.27507157,0.12437936){\color[rgb]{0,0,0}\makebox(0,0)[lt]{\lineheight{0}\smash{\begin{tabular}[t]{l}3\end{tabular}}}}%
    \put(0,0){\includegraphics[width=\unitlength,page=2]{figure11.pdf}}%
    \put(0.83010275,0.12151539){\makebox(0,0)[lt]{\lineheight{1.25}\smash{\begin{tabular}[t]{l}2\end{tabular}}}}%
    \put(0.9045579,0.04399299){\makebox(0,0)[lt]{\lineheight{1.25}\smash{\begin{tabular}[t]{l}1\end{tabular}}}}%
    \put(0.97901289,0.12151539){\makebox(0,0)[lt]{\lineheight{1.25}\smash{\begin{tabular}[t]{l}2\end{tabular}}}}%
    \put(0,0){\includegraphics[width=\unitlength,page=3]{figure11.pdf}}%
    \put(0.79425409,0.19841826){\makebox(0,0)[lt]{\lineheight{1.25}\smash{\begin{tabular}[t]{l}3\end{tabular}}}}%
    \put(0.79425409,0.04399299){\makebox(0,0)[lt]{\lineheight{1.25}\smash{\begin{tabular}[t]{l}3\end{tabular}}}}%
    \put(0,0){\includegraphics[width=\unitlength,page=4]{figure11.pdf}}%
    \put(0.60949509,0.12151539){\makebox(0,0)[lt]{\lineheight{1.25}\smash{\begin{tabular}[t]{l}2\end{tabular}}}}%
    \put(0.68395028,0.19841826){\makebox(0,0)[lt]{\lineheight{1.25}\smash{\begin{tabular}[t]{l}1\end{tabular}}}}%
    \put(0.68395028,0.04399299){\makebox(0,0)[lt]{\lineheight{1.25}\smash{\begin{tabular}[t]{l}1\end{tabular}}}}%
    \put(0,0){\includegraphics[width=\unitlength,page=5]{figure11.pdf}}%
    \put(0.83010275,0.3421229){\makebox(0,0)[lt]{\lineheight{1.25}\smash{\begin{tabular}[t]{l}2\end{tabular}}}}%
    \put(0.9045579,0.41902584){\makebox(0,0)[lt]{\lineheight{1.25}\smash{\begin{tabular}[t]{l}1\end{tabular}}}}%
    \put(0.9045579,0.26460052){\makebox(0,0)[lt]{\lineheight{1.25}\smash{\begin{tabular}[t]{l}1\end{tabular}}}}%
    \put(0.97901289,0.3421229){\makebox(0,0)[lt]{\lineheight{1.25}\smash{\begin{tabular}[t]{l}2\end{tabular}}}}%
    \put(0,0){\includegraphics[width=\unitlength,page=6]{figure11.pdf}}%
    \put(0.83010283,0.23181915){\makebox(0,0)[lt]{\lineheight{1.25}\smash{\begin{tabular}[t]{l}3\end{tabular}}}}%
    \put(0.97901289,0.23181915){\makebox(0,0)[lt]{\lineheight{1.25}\smash{\begin{tabular}[t]{l}3\end{tabular}}}}%
    \put(0.71979894,0.12151539){\makebox(0,0)[lt]{\lineheight{1.25}\smash{\begin{tabular}[t]{l}2\end{tabular}}}}%
    \put(0.9045579,0.15429676){\makebox(0,0)[lt]{\lineheight{1.25}\smash{\begin{tabular}[t]{l}1\end{tabular}}}}%
  \end{picture}%
\endgroup%

%% file: tiles.pdf_tex
%% Creator: Inkscape 1.0 (4035a4f, 2020-05-01), www.inkscape.org
%% PDF/EPS/PS + LaTeX output extension by Johan Engelen, 2010
%% Accompanies image file 'tiles.pdf' (pdf, eps, ps)
%%
%% To include the image in your LaTeX document, write
%%   \input{<filename>.pdf_tex}
%%  instead of
%%   \includegraphics{<filename>.pdf}
%% To scale the image, write
%%   \def\svgwidth{<desired width>}
%%   \input{<filename>.pdf_tex}
%%  instead of
%%   \includegraphics[width=<desired width>]{<filename>.pdf}
%%
%% Images with a different path to the parent latex file can
%% be accessed with the `import' package (which may need to be
%% installed) using
%%   \usepackage{import}
%% in the preamble, and then including the image with
%%   \import{<path to file>}{<filename>.pdf_tex}
%% Alternatively, one can specify
%%   \graphicspath{{<path to file>/}}
%% 
%% For more information, please see info/svg-inkscape on CTAN:
%%   http://tug.ctan.org/tex-archive/info/svg-inkscape
%%
\begingroup%
  \makeatletter%
  \providecommand\color[2][]{%
    \errmessage{(Inkscape) Color is used for the text in Inkscape, but the package 'color.sty' is not loaded}%
    \renewcommand\color[2][]{}%
  }%
  \providecommand\transparent[1]{%
    \errmessage{(Inkscape) Transparency is used (non-zero) for the text in Inkscape, but the package 'transparent.sty' is not loaded}%
    \renewcommand\transparent[1]{}%
  }%
  \providecommand\rotatebox[2]{#2}%
  \newcommand*\fsize{\dimexpr\f@size pt\relax}%
  \newcommand*\lineheight[1]{\fontsize{\fsize}{#1\fsize}\selectfont}%
  \ifx\svgwidth\undefined%
    \setlength{\unitlength}{361.68835245bp}%
    \ifx\svgscale\undefined%
      \relax%
    \else%
      \setlength{\unitlength}{\unitlength * \real{\svgscale}}%
    \fi%
  \else%
    \setlength{\unitlength}{\svgwidth}%
  \fi%
  \global\let\svgwidth\undefined%
  \global\let\svgscale\undefined%
  \makeatother%
  \begin{picture}(1,0.18690208)%
    \lineheight{1}%
    \setlength\tabcolsep{0pt}%
    \put(0,0){\includegraphics[width=\unitlength,page=1]{tiles.pdf}}%
    \put(-0.0022545,0.10855394){\makebox(0,0)[lt]{\lineheight{1.25}\smash{\begin{tabular}[t]{l}3\end{tabular}}}}%
    \put(0.05373292,0.16638204){\makebox(0,0)[lt]{\lineheight{1.25}\smash{\begin{tabular}[t]{l}2\end{tabular}}}}%
    \put(0.05373292,0.05025998){\makebox(0,0)[lt]{\lineheight{1.25}\smash{\begin{tabular}[t]{l}2\end{tabular}}}}%
    \put(0.10972032,0.10855394){\makebox(0,0)[lt]{\lineheight{1.25}\smash{\begin{tabular}[t]{l}3\end{tabular}}}}%
    \put(0.04543847,0.00487351){\makebox(0,0)[lt]{\lineheight{1.25}\smash{\begin{tabular}[t]{l}$G_1^+$\end{tabular}}}}%
    \put(0,0){\includegraphics[width=\unitlength,page=2]{tiles.pdf}}%
    \put(0.16363416,0.10855394){\makebox(0,0)[lt]{\lineheight{1.25}\smash{\begin{tabular}[t]{l}2\end{tabular}}}}%
    \put(0.21962161,0.16638204){\makebox(0,0)[lt]{\lineheight{1.25}\smash{\begin{tabular}[t]{l}3\end{tabular}}}}%
    \put(0.21962161,0.05025998){\makebox(0,0)[lt]{\lineheight{1.25}\smash{\begin{tabular}[t]{l}3\end{tabular}}}}%
    \put(0.27560903,0.10855394){\makebox(0,0)[lt]{\lineheight{1.25}\smash{\begin{tabular}[t]{l}2\end{tabular}}}}%
    \put(0.21132717,0.00487351){\makebox(0,0)[lt]{\lineheight{1.25}\smash{\begin{tabular}[t]{l}$G_1^-$\end{tabular}}}}%
    \put(0,0){\includegraphics[width=\unitlength,page=3]{tiles.pdf}}%
    \put(0.32952287,0.10855394){\makebox(0,0)[lt]{\lineheight{1.25}\smash{\begin{tabular}[t]{l}1\end{tabular}}}}%
    \put(0.38551032,0.16638204){\makebox(0,0)[lt]{\lineheight{1.25}\smash{\begin{tabular}[t]{l}3\end{tabular}}}}%
    \put(0.38551032,0.05025998){\makebox(0,0)[lt]{\lineheight{1.25}\smash{\begin{tabular}[t]{l}3\end{tabular}}}}%
    \put(0.44149765,0.10855394){\makebox(0,0)[lt]{\lineheight{1.25}\smash{\begin{tabular}[t]{l}1\end{tabular}}}}%
    \put(0.37721588,0.00487351){\makebox(0,0)[lt]{\lineheight{1.25}\smash{\begin{tabular}[t]{l}$G_2^+$\end{tabular}}}}%
    \put(0,0){\includegraphics[width=\unitlength,page=4]{tiles.pdf}}%
    \put(0.49541146,0.10855394){\makebox(0,0)[lt]{\lineheight{1.25}\smash{\begin{tabular}[t]{l}3\end{tabular}}}}%
    \put(0.55139891,0.16638204){\makebox(0,0)[lt]{\lineheight{1.25}\smash{\begin{tabular}[t]{l}1\end{tabular}}}}%
    \put(0.55139891,0.05025998){\makebox(0,0)[lt]{\lineheight{1.25}\smash{\begin{tabular}[t]{l}1\end{tabular}}}}%
    \put(0.60738624,0.10855394){\makebox(0,0)[lt]{\lineheight{1.25}\smash{\begin{tabular}[t]{l}3\end{tabular}}}}%
    \put(0.54310447,0.00487351){\makebox(0,0)[lt]{\lineheight{1.25}\smash{\begin{tabular}[t]{l}$G_2^-$\end{tabular}}}}%
    \put(0,0){\includegraphics[width=\unitlength,page=5]{tiles.pdf}}%
    \put(0.66129999,0.10855394){\makebox(0,0)[lt]{\lineheight{1.25}\smash{\begin{tabular}[t]{l}2\end{tabular}}}}%
    \put(0.7172875,0.16638204){\makebox(0,0)[lt]{\lineheight{1.25}\smash{\begin{tabular}[t]{l}1\end{tabular}}}}%
    \put(0.7172875,0.05025998){\makebox(0,0)[lt]{\lineheight{1.25}\smash{\begin{tabular}[t]{l}1\end{tabular}}}}%
    \put(0.77327489,0.10855394){\makebox(0,0)[lt]{\lineheight{1.25}\smash{\begin{tabular}[t]{l}2\end{tabular}}}}%
    \put(0.70899306,0.00487351){\makebox(0,0)[lt]{\lineheight{1.25}\smash{\begin{tabular}[t]{l}$G_3^+$\end{tabular}}}}%
    \put(0,0){\includegraphics[width=\unitlength,page=6]{tiles.pdf}}%
    \put(0.82718865,0.10855394){\makebox(0,0)[lt]{\lineheight{1.25}\smash{\begin{tabular}[t]{l}1\end{tabular}}}}%
    \put(0.88317621,0.16638204){\makebox(0,0)[lt]{\lineheight{1.25}\smash{\begin{tabular}[t]{l}2\end{tabular}}}}%
    \put(0.88317621,0.05025998){\makebox(0,0)[lt]{\lineheight{1.25}\smash{\begin{tabular}[t]{l}2\end{tabular}}}}%
    \put(0.93916367,0.10855394){\makebox(0,0)[lt]{\lineheight{1.25}\smash{\begin{tabular}[t]{l}1\end{tabular}}}}%
    \put(0.87488177,0.00487351){\makebox(0,0)[lt]{\lineheight{1.25}\smash{\begin{tabular}[t]{l}$G_3^-$\end{tabular}}}}%
  \end{picture}%
\endgroup%

%% file: bracelet.pdf_tex
%% Creator: Inkscape 1.0 (4035a4f, 2020-05-01), www.inkscape.org
%% PDF/EPS/PS + LaTeX output extension by Johan Engelen, 2010
%% Accompanies image file 'bracelet.pdf' (pdf, eps, ps)
%%
%% To include the image in your LaTeX document, write
%%   \input{<filename>.pdf_tex}
%%  instead of
%%   \includegraphics{<filename>.pdf}
%% To scale the image, write
%%   \def\svgwidth{<desired width>}
%%   \input{<filename>.pdf_tex}
%%  instead of
%%   \includegraphics[width=<desired width>]{<filename>.pdf}
%%
%% Images with a different path to the parent latex file can
%% be accessed with the `import' package (which may need to be
%% installed) using
%%   \usepackage{import}
%% in the preamble, and then including the image with
%%   \import{<path to file>}{<filename>.pdf_tex}
%% Alternatively, one can specify
%%   \graphicspath{{<path to file>/}}
%% 
%% For more information, please see info/svg-inkscape on CTAN:
%%   http://tug.ctan.org/tex-archive/info/svg-inkscape
%%
\begingroup%
  \makeatletter%
  \providecommand\color[2][]{%
    \errmessage{(Inkscape) Color is used for the text in Inkscape, but the package 'color.sty' is not loaded}%
    \renewcommand\color[2][]{}%
  }%
  \providecommand\transparent[1]{%
    \errmessage{(Inkscape) Transparency is used (non-zero) for the text in Inkscape, but the package 'transparent.sty' is not loaded}%
    \renewcommand\transparent[1]{}%
  }%
  \providecommand\rotatebox[2]{#2}%
  \newcommand*\fsize{\dimexpr\f@size pt\relax}%
  \newcommand*\lineheight[1]{\fontsize{\fsize}{#1\fsize}\selectfont}%
  \ifx\svgwidth\undefined%
    \setlength{\unitlength}{156.26865455bp}%
    \ifx\svgscale\undefined%
      \relax%
    \else%
      \setlength{\unitlength}{\unitlength * \real{\svgscale}}%
    \fi%
  \else%
    \setlength{\unitlength}{\svgwidth}%
  \fi%
  \global\let\svgwidth\undefined%
  \global\let\svgscale\undefined%
  \makeatother%
  \begin{picture}(1,0.87038055)%
    \lineheight{1}%
    \setlength\tabcolsep{0pt}%
    \put(0,0){\includegraphics[width=\unitlength,page=1]{bracelet.pdf}}%
  \end{picture}%
\endgroup%

%% file: figure21.pdf_tex
%% Creator: Inkscape inkscape 0.48.5, www.inkscape.org
%% PDF/EPS/PS + LaTeX output extension by Johan Engelen, 2010
%% Accompanies image file 'figure21.pdf' (pdf, eps, ps)
%%
%% To include the image in your LaTeX document, write
%%   \input{<filename>.pdf_tex}
%%  instead of
%%   \includegraphics{<filename>.pdf}
%% To scale the image, write
%%   \def\svgwidth{<desired width>}
%%   \input{<filename>.pdf_tex}
%%  instead of
%%   \includegraphics[width=<desired width>]{<filename>.pdf}
%%
%% Images with a different path to the parent latex file can
%% be accessed with the `import' package (which may need to be
%% installed) using
%%   \usepackage{import}
%% in the preamble, and then including the image with
%%   \import{<path to file>}{<filename>.pdf_tex}
%% Alternatively, one can specify
%%   \graphicspath{{<path to file>/}}
%% 
%% For more information, please see info/svg-inkscape on CTAN:
%%   http://tug.ctan.org/tex-archive/info/svg-inkscape
%%
\begingroup%
  \makeatletter%
  \providecommand\color[2][]{%
    \errmessage{(Inkscape) Color is used for the text in Inkscape, but the package 'color.sty' is not loaded}%
    \renewcommand\color[2][]{}%
  }%
  \providecommand\transparent[1]{%
    \errmessage{(Inkscape) Transparency is used (non-zero) for the text in Inkscape, but the package 'transparent.sty' is not loaded}%
    \renewcommand\transparent[1]{}%
  }%
  \providecommand\rotatebox[2]{#2}%
  \ifx\svgwidth\undefined%
    \setlength{\unitlength}{410.42614746bp}%
    \ifx\svgscale\undefined%
      \relax%
    \else%
      \setlength{\unitlength}{\unitlength * \real{\svgscale}}%
    \fi%
  \else%
    \setlength{\unitlength}{\svgwidth}%
  \fi%
  \global\let\svgwidth\undefined%
  \global\let\svgscale\undefined%
  \makeatother%
  \begin{picture}(1,0.27477538)%
    \put(0,0){\includegraphics[width=\unitlength]{figure21.pdf}}%
    \put(0.00067765,0.0009746){\color[rgb]{0,0,0}\makebox(0,0)[lb]{\smash{A}}}%
    \put(0.5161221,0.25679483){\color[rgb]{0,0,0}\makebox(0,0)[lb]{\smash{B}}}%
    \put(0.23842879,0.086565){\color[rgb]{0,0,0.03137255}\makebox(0,0)[lb]{\smash{$\color{blue}\zgr$}}}%
    \put(0.19373157,0.13126223){\color[rgb]{0,0,0}\makebox(0,0)[lb]{\smash{$\color{red}\zgl$}}}%
    \put(0.76897214,0.12938495){\color[rgb]{0,0,0}\makebox(0,0)[lb]{\smash{1}}}%
    \put(0.79770478,0.18295181){\color[rgb]{0,0,0}\makebox(0,0)[lb]{\smash{3}}}%
    \put(0.85517007,0.20006722){\color[rgb]{0,0,0}\makebox(0,0)[lb]{\smash{2}}}%
    \put(0.83177974,0.04413173){\color[rgb]{0,0,0}\makebox(0,0)[lb]{\smash{2}}}%
    \put(0.88736769,0.05040665){\color[rgb]{0,0,0}\makebox(0,0)[lb]{\smash{3}}}%
    \put(0.92490763,0.10209624){\color[rgb]{0,0,0}\makebox(0,0)[lb]{\smash{1}}}%
    \put(0.77640618,0.05927629){\color[rgb]{0,0,0.03137255}\makebox(0,0)[lb]{\smash{$\color{blue}\zgr$}}}%
    \put(0.7395057,0.09617674){\color[rgb]{0,0,0}\makebox(0,0)[lb]{\smash{$\color{red}\zgl$}}}%
  \end{picture}%
\endgroup%

%% file: figpiecewisepath.pdf_tex
%% Creator: Inkscape 1.0 (4035a4f, 2020-05-01), www.inkscape.org
%% PDF/EPS/PS + LaTeX output extension by Johan Engelen, 2010
%% Accompanies image file 'figpiecewisepath.pdf' (pdf, eps, ps)
%%
%% To include the image in your LaTeX document, write
%%   \input{<filename>.pdf_tex}
%%  instead of
%%   \includegraphics{<filename>.pdf}
%% To scale the image, write
%%   \def\svgwidth{<desired width>}
%%   \input{<filename>.pdf_tex}
%%  instead of
%%   \includegraphics[width=<desired width>]{<filename>.pdf}
%%
%% Images with a different path to the parent latex file can
%% be accessed with the `import' package (which may need to be
%% installed) using
%%   \usepackage{import}
%% in the preamble, and then including the image with
%%   \import{<path to file>}{<filename>.pdf_tex}
%% Alternatively, one can specify
%%   \graphicspath{{<path to file>/}}
%% 
%% For more information, please see info/svg-inkscape on CTAN:
%%   http://tug.ctan.org/tex-archive/info/svg-inkscape
%%
\begingroup%
  \makeatletter%
  \providecommand\color[2][]{%
    \errmessage{(Inkscape) Color is used for the text in Inkscape, but the package 'color.sty' is not loaded}%
    \renewcommand\color[2][]{}%
  }%
  \providecommand\transparent[1]{%
    \errmessage{(Inkscape) Transparency is used (non-zero) for the text in Inkscape, but the package 'transparent.sty' is not loaded}%
    \renewcommand\transparent[1]{}%
  }%
  \providecommand\rotatebox[2]{#2}%
  \newcommand*\fsize{\dimexpr\f@size pt\relax}%
  \newcommand*\lineheight[1]{\fontsize{\fsize}{#1\fsize}\selectfont}%
  \ifx\svgwidth\undefined%
    \setlength{\unitlength}{378.77272964bp}%
    \ifx\svgscale\undefined%
      \relax%
    \else%
      \setlength{\unitlength}{\unitlength * \real{\svgscale}}%
    \fi%
  \else%
    \setlength{\unitlength}{\svgwidth}%
  \fi%
  \global\let\svgwidth\undefined%
  \global\let\svgscale\undefined%
  \makeatother%
  \begin{picture}(1,0.29294193)%
    \lineheight{1}%
    \setlength\tabcolsep{0pt}%
    \put(0,0){\includegraphics[width=\unitlength,page=1]{figpiecewisepath.pdf}}%
    \put(0.06858455,0.00465369){\makebox(0,0)[lt]{\lineheight{1.25}\smash{\begin{tabular}[t]{l}$A$\end{tabular}}}}%
    \put(0.42833914,0.16306064){\makebox(0,0)[lt]{\lineheight{1.25}\smash{\begin{tabular}[t]{l}$B$\end{tabular}}}}%
    \put(0,0){\includegraphics[width=\unitlength,page=2]{figpiecewisepath.pdf}}%
    \put(0.16488983,0.27262553){\makebox(0,0)[lt]{\lineheight{1.25}\smash{\begin{tabular}[t]{l}$Q_1$\end{tabular}}}}%
    \put(0.18097855,0.15029841){\makebox(0,0)[lt]{\lineheight{1.25}\smash{\begin{tabular}[t]{l}$Q_2$\end{tabular}}}}%
    \put(0.23259104,0.1887344){\makebox(0,0)[lt]{\lineheight{1.25}\smash{\begin{tabular}[t]{l}$Q_3$\end{tabular}}}}%
    \put(0.42575686,0.26285739){\makebox(0,0)[lt]{\lineheight{1.25}\smash{\begin{tabular}[t]{l}$Q_4$\end{tabular}}}}%
    \put(0,0){\includegraphics[width=\unitlength,page=3]{figpiecewisepath.pdf}}%
    \put(0.50420201,0.00465369){\makebox(0,0)[lt]{\lineheight{1.25}\smash{\begin{tabular}[t]{l}$A$\end{tabular}}}}%
    \put(0.86395647,0.16306064){\makebox(0,0)[lt]{\lineheight{1.25}\smash{\begin{tabular}[t]{l}$B$\end{tabular}}}}%
    \put(0,0){\includegraphics[width=\unitlength,page=4]{figpiecewisepath.pdf}}%
    \put(0.59654711,0.26866537){\makebox(0,0)[lt]{\lineheight{1.25}\smash{\begin{tabular}[t]{l}$\theta_1$\end{tabular}}}}%
    \put(0.58887487,0.14633824){\makebox(0,0)[lt]{\lineheight{1.25}\smash{\begin{tabular}[t]{l}$\theta_2$\end{tabular}}}}%
    \put(0.68008896,0.14517262){\makebox(0,0)[lt]{\lineheight{1.25}\smash{\begin{tabular}[t]{l}$\theta_3$\end{tabular}}}}%
    \put(0.85345384,0.26285739){\makebox(0,0)[lt]{\lineheight{1.25}\smash{\begin{tabular}[t]{l}$\theta_4$\end{tabular}}}}%
    \put(0,0){\includegraphics[width=\unitlength,page=5]{figpiecewisepath.pdf}}%
    \put(-0.00199812,0.26788335){\makebox(0,0)[lt]{\lineheight{1.25}\smash{\begin{tabular}[t]{l}$ $\end{tabular}}}}%
  \end{picture}%
\endgroup%

%% file: figure28case1AB.pdf_tex
%% Creator: Inkscape 1.0 (4035a4f, 2020-05-01), www.inkscape.org
%% PDF/EPS/PS + LaTeX output extension by Johan Engelen, 2010
%% Accompanies image file 'figure28case1AB.pdf' (pdf, eps, ps)
%%
%% To include the image in your LaTeX document, write
%%   \input{<filename>.pdf_tex}
%%  instead of
%%   \includegraphics{<filename>.pdf}
%% To scale the image, write
%%   \def\svgwidth{<desired width>}
%%   \input{<filename>.pdf_tex}
%%  instead of
%%   \includegraphics[width=<desired width>]{<filename>.pdf}
%%
%% Images with a different path to the parent latex file can
%% be accessed with the `import' package (which may need to be
%% installed) using
%%   \usepackage{import}
%% in the preamble, and then including the image with
%%   \import{<path to file>}{<filename>.pdf_tex}
%% Alternatively, one can specify
%%   \graphicspath{{<path to file>/}}
%% 
%% For more information, please see info/svg-inkscape on CTAN:
%%   http://tug.ctan.org/tex-archive/info/svg-inkscape
%%
\begingroup%
  \makeatletter%
  \providecommand\color[2][]{%
    \errmessage{(Inkscape) Color is used for the text in Inkscape, but the package 'color.sty' is not loaded}%
    \renewcommand\color[2][]{}%
  }%
  \providecommand\transparent[1]{%
    \errmessage{(Inkscape) Transparency is used (non-zero) for the text in Inkscape, but the package 'transparent.sty' is not loaded}%
    \renewcommand\transparent[1]{}%
  }%
  \providecommand\rotatebox[2]{#2}%
  \newcommand*\fsize{\dimexpr\f@size pt\relax}%
  \newcommand*\lineheight[1]{\fontsize{\fsize}{#1\fsize}\selectfont}%
  \ifx\svgwidth\undefined%
    \setlength{\unitlength}{509.25815782bp}%
    \ifx\svgscale\undefined%
      \relax%
    \else%
      \setlength{\unitlength}{\unitlength * \real{\svgscale}}%
    \fi%
  \else%
    \setlength{\unitlength}{\svgwidth}%
  \fi%
  \global\let\svgwidth\undefined%
  \global\let\svgscale\undefined%
  \makeatother%
  \begin{picture}(1,0.55522154)%
    \lineheight{1}%
    \setlength\tabcolsep{0pt}%
    \put(0,0){\includegraphics[width=\unitlength,page=1]{figure28case1AB.pdf}}%
    \put(0.0069616,0.39500786){\color[rgb]{0,0,0}\makebox(0,0)[lt]{\lineheight{0}\smash{\begin{tabular}[t]{l}$\gamma$\end{tabular}}}}%
    \put(0.1178325,0.14793211){\color[rgb]{0,0,0}\makebox(0,0)[lt]{\lineheight{0}\smash{\begin{tabular}[t]{l}B\end{tabular}}}}%
    \put(0.4154119,0.30013701){\color[rgb]{0,0,0}\makebox(0,0)[lt]{\lineheight{0}\smash{\begin{tabular}[t]{l}A\end{tabular}}}}%
    \put(0.16424142,0.30331546){\color[rgb]{0,0,0}\makebox(0,0)[lt]{\lineheight{0}\smash{\begin{tabular}[t]{l}$Q_s$\end{tabular}}}}%
    \put(0.06381615,0.29660247){\color[rgb]{0,0,0}\makebox(0,0)[lt]{\lineheight{0}\smash{\begin{tabular}[t]{l}$D=Q_{s+1}$\end{tabular}}}}%
    \put(0,0){\includegraphics[width=\unitlength,page=2]{figure28case1AB.pdf}}%
    \put(0.21958641,0.25718526){\color[rgb]{0,0,0}\makebox(0,0)[lt]{\lineheight{0}\smash{\begin{tabular}[t]{l}$\gamma_{AD}^L$\end{tabular}}}}%
    \put(0.2385779,0.30249568){\color[rgb]{0,0,0}\makebox(0,0)[lt]{\lineheight{0}\smash{\begin{tabular}[t]{l}$Q_{s-1}$\end{tabular}}}}%
    \put(0.33363032,0.30617751){\color[rgb]{0,0,0}\makebox(0,0)[lt]{\lineheight{0}\smash{\begin{tabular}[t]{l}...\end{tabular}}}}%
    \put(0,0){\includegraphics[width=\unitlength,page=3]{figure28case1AB.pdf}}%
    \put(0.53726259,0.54088543){\color[rgb]{0,0,0}\makebox(0,0)[lt]{\lineheight{0}\smash{\begin{tabular}[t]{l}$\gamma_1$\end{tabular}}}}%
    \put(0.64813347,0.29380973){\color[rgb]{0,0,0}\makebox(0,0)[lt]{\lineheight{0}\smash{\begin{tabular}[t]{l}B\end{tabular}}}}%
    \put(0.94571284,0.44601459){\color[rgb]{0,0,0}\makebox(0,0)[lt]{\lineheight{0}\smash{\begin{tabular}[t]{l}A\end{tabular}}}}%
    \put(0.70043329,0.44919303){\color[rgb]{0,0,0}\makebox(0,0)[lt]{\lineheight{0}\smash{\begin{tabular}[t]{l}$Q_s$\end{tabular}}}}%
    \put(0.62946267,0.40183267){\color[rgb]{0,0,0}\makebox(0,0)[lt]{\lineheight{0}\smash{\begin{tabular}[t]{l}$D$\end{tabular}}}}%
    \put(0,0){\includegraphics[width=\unitlength,page=4]{figure28case1AB.pdf}}%
    \put(0.5328444,0.24781207){\color[rgb]{0,0,0}\makebox(0,0)[lt]{\lineheight{0}\smash{\begin{tabular}[t]{l}$\beta$\end{tabular}}}}%
    \put(0.64371528,0.00073637){\color[rgb]{0,0,0}\makebox(0,0)[lt]{\lineheight{0}\smash{\begin{tabular}[t]{l}B\end{tabular}}}}%
    \put(0.94129465,0.15294122){\color[rgb]{0,0,0}\makebox(0,0)[lt]{\lineheight{0}\smash{\begin{tabular}[t]{l}A\end{tabular}}}}%
    \put(0.69601509,0.15611968){\color[rgb]{0,0,0}\makebox(0,0)[lt]{\lineheight{0}\smash{\begin{tabular}[t]{l}$Q_s$\end{tabular}}}}%
    \put(0.62504448,0.11170477){\color[rgb]{0,0,0}\makebox(0,0)[lt]{\lineheight{0}\smash{\begin{tabular}[t]{l}$D$\end{tabular}}}}%
    \put(0,0){\includegraphics[width=\unitlength,page=5]{figure28case1AB.pdf}}%
    \put(0.82205112,0.15711681){\color[rgb]{0,0,0}\makebox(0,0)[lt]{\lineheight{0}\smash{\begin{tabular}[t]{l}$\epsilon$\end{tabular}}}}%
    \put(0,0){\includegraphics[width=\unitlength,page=6]{figure28case1AB.pdf}}%
    \put(0.81715131,0.45100601){\color[rgb]{0,0,0}\makebox(0,0)[lt]{\lineheight{0}\smash{\begin{tabular}[t]{l}$\gamma_{AD}^R$\end{tabular}}}}%
    \put(0.11673169,0.26024702){\color[rgb]{0,0,0}\makebox(0,0)[lt]{\lineheight{0}\smash{\begin{tabular}[t]{l}$P$\end{tabular}}}}%
  \end{picture}%
\endgroup%

%% file: figure28case2AB.pdf_tex
%% Creator: Inkscape 1.0 (4035a4f, 2020-05-01), www.inkscape.org
%% PDF/EPS/PS + LaTeX output extension by Johan Engelen, 2010
%% Accompanies image file 'figure28case2AB.pdf' (pdf, eps, ps)
%%
%% To include the image in your LaTeX document, write
%%   \input{<filename>.pdf_tex}
%%  instead of
%%   \includegraphics{<filename>.pdf}
%% To scale the image, write
%%   \def\svgwidth{<desired width>}
%%   \input{<filename>.pdf_tex}
%%  instead of
%%   \includegraphics[width=<desired width>]{<filename>.pdf}
%%
%% Images with a different path to the parent latex file can
%% be accessed with the `import' package (which may need to be
%% installed) using
%%   \usepackage{import}
%% in the preamble, and then including the image with
%%   \import{<path to file>}{<filename>.pdf_tex}
%% Alternatively, one can specify
%%   \graphicspath{{<path to file>/}}
%% 
%% For more information, please see info/svg-inkscape on CTAN:
%%   http://tug.ctan.org/tex-archive/info/svg-inkscape
%%
\begingroup%
  \makeatletter%
  \providecommand\color[2][]{%
    \errmessage{(Inkscape) Color is used for the text in Inkscape, but the package 'color.sty' is not loaded}%
    \renewcommand\color[2][]{}%
  }%
  \providecommand\transparent[1]{%
    \errmessage{(Inkscape) Transparency is used (non-zero) for the text in Inkscape, but the package 'transparent.sty' is not loaded}%
    \renewcommand\transparent[1]{}%
  }%
  \providecommand\rotatebox[2]{#2}%
  \newcommand*\fsize{\dimexpr\f@size pt\relax}%
  \newcommand*\lineheight[1]{\fontsize{\fsize}{#1\fsize}\selectfont}%
  \ifx\svgwidth\undefined%
    \setlength{\unitlength}{494.95770264bp}%
    \ifx\svgscale\undefined%
      \relax%
    \else%
      \setlength{\unitlength}{\unitlength * \real{\svgscale}}%
    \fi%
  \else%
    \setlength{\unitlength}{\svgwidth}%
  \fi%
  \global\let\svgwidth\undefined%
  \global\let\svgscale\undefined%
  \makeatother%
  \begin{picture}(1,0.41596871)%
    \lineheight{1}%
    \setlength\tabcolsep{0pt}%
    \put(0,0){\includegraphics[width=\unitlength,page=1]{figure28case2AB.pdf}}%
    \put(-0.00061854,0.19901871){\color[rgb]{0,0,0}\makebox(0,0)[lt]{\lineheight{0}\smash{\begin{tabular}[t]{l}$\gamma$\end{tabular}}}}%
    \put(0.10739452,0.06299635){\color[rgb]{0,0,0}\makebox(0,0)[lt]{\lineheight{0}\smash{\begin{tabular}[t]{l}$B$\end{tabular}}}}%
    \put(0.41357164,0.21959875){\color[rgb]{0,0,0}\makebox(0,0)[lt]{\lineheight{0}\smash{\begin{tabular}[t]{l}$A$\end{tabular}}}}%
    \put(0.14605259,0.22286903){\color[rgb]{0,0,0}\makebox(0,0)[lt]{\lineheight{0}\smash{\begin{tabular}[t]{l}$D\!=\!Q_{s+1}$\end{tabular}}}}%
    \put(0.08666895,0.21656819){\color[rgb]{0,0,0}\makebox(0,0)[lt]{\lineheight{0}\smash{\begin{tabular}[t]{l}$E$\end{tabular}}}}%
    \put(0,0){\includegraphics[width=\unitlength,page=2]{figure28case2AB.pdf}}%
    \put(0.21208829,0.17540612){\color[rgb]{0,0,0}\makebox(0,0)[lt]{\lineheight{0}\smash{\begin{tabular}[t]{l}$\gamma_{AE}^L$\end{tabular}}}}%
    \put(0.24981188,0.22202557){\color[rgb]{0,0,0}\makebox(0,0)[lt]{\lineheight{0}\smash{\begin{tabular}[t]{l}$Q_{s}$\end{tabular}}}}%
    \put(0.32942721,0.22581379){\color[rgb]{0,0,0}\makebox(0,0)[lt]{\lineheight{0}\smash{\begin{tabular}[t]{l}...\end{tabular}}}}%
    \put(0,0){\includegraphics[width=\unitlength,page=3]{figure28case2AB.pdf}}%
    \put(0.54177164,0.13998518){\color[rgb]{0,0,0}\makebox(0,0)[lt]{\lineheight{0}\smash{\begin{tabular}[t]{l}$\beta$\end{tabular}}}}%
    \put(0.65584583,0.00396281){\color[rgb]{0,0,0}\makebox(0,0)[lt]{\lineheight{0}\smash{\begin{tabular}[t]{l}$B$\end{tabular}}}}%
    \put(0.96202272,0.16056522){\color[rgb]{0,0,0}\makebox(0,0)[lt]{\lineheight{0}\smash{\begin{tabular}[t]{l}$A$\end{tabular}}}}%
    \put(0.70965666,0.1638355){\color[rgb]{0,0,0}\makebox(0,0)[lt]{\lineheight{0}\smash{\begin{tabular}[t]{l}$D$\end{tabular}}}}%
    \put(0.63663559,0.1545041){\color[rgb]{0,0,0}\makebox(0,0)[lt]{\lineheight{0}\smash{\begin{tabular}[t]{l}$E$\end{tabular}}}}%
    \put(0,0){\includegraphics[width=\unitlength,page=4]{figure28case2AB.pdf}}%
    \put(0.8363035,0.16486144){\color[rgb]{0,0,0}\makebox(0,0)[lt]{\lineheight{0}\smash{\begin{tabular}[t]{l}$\epsilon$\end{tabular}}}}%
    \put(0,0){\includegraphics[width=\unitlength,page=5]{figure28case2AB.pdf}}%
    \put(0.54468419,0.3750606){\color[rgb]{0,0,0}\makebox(0,0)[lt]{\lineheight{0}\smash{\begin{tabular}[t]{l}$\gamma_1$\end{tabular}}}}%
    \put(0.65875838,0.23903826){\color[rgb]{0,0,0}\makebox(0,0)[lt]{\lineheight{0}\smash{\begin{tabular}[t]{l}$B$\end{tabular}}}}%
    \put(0.96493545,0.3600545){\color[rgb]{0,0,0}\makebox(0,0)[lt]{\lineheight{0}\smash{\begin{tabular}[t]{l}$A$\end{tabular}}}}%
    \put(0.71067516,0.39588037){\color[rgb]{0,0,0}\makebox(0,0)[lt]{\lineheight{0}\smash{\begin{tabular}[t]{l}$D$\end{tabular}}}}%
    \put(0.63803281,0.39261009){\color[rgb]{0,0,0}\makebox(0,0)[lt]{\lineheight{0}\smash{\begin{tabular}[t]{l}$E$\end{tabular}}}}%
    \put(0,0){\includegraphics[width=\unitlength,page=6]{figure28case2AB.pdf}}%
    \put(0.78446228,0.4012184){\color[rgb]{0,0,0}\makebox(0,0)[lt]{\lineheight{0}\smash{\begin{tabular}[t]{l}$\gamma_{AE}^R$\end{tabular}}}}%
    \put(0.14553932,0.17229131){\color[rgb]{0,0,0}\makebox(0,0)[lt]{\lineheight{0}\smash{\begin{tabular}[t]{l}$P$\end{tabular}}}}%
  \end{picture}%
\endgroup%

%% file: figure28case3_1AB.pdf_tex
%% Creator: Inkscape 1.0 (4035a4f, 2020-05-01), www.inkscape.org
%% PDF/EPS/PS + LaTeX output extension by Johan Engelen, 2010
%% Accompanies image file 'figure28case3_1AB.pdf' (pdf, eps, ps)
%%
%% To include the image in your LaTeX document, write
%%   \input{<filename>.pdf_tex}
%%  instead of
%%   \includegraphics{<filename>.pdf}
%% To scale the image, write
%%   \def\svgwidth{<desired width>}
%%   \input{<filename>.pdf_tex}
%%  instead of
%%   \includegraphics[width=<desired width>]{<filename>.pdf}
%%
%% Images with a different path to the parent latex file can
%% be accessed with the `import' package (which may need to be
%% installed) using
%%   \usepackage{import}
%% in the preamble, and then including the image with
%%   \import{<path to file>}{<filename>.pdf_tex}
%% Alternatively, one can specify
%%   \graphicspath{{<path to file>/}}
%% 
%% For more information, please see info/svg-inkscape on CTAN:
%%   http://tug.ctan.org/tex-archive/info/svg-inkscape
%%
\begingroup%
  \makeatletter%
  \providecommand\color[2][]{%
    \errmessage{(Inkscape) Color is used for the text in Inkscape, but the package 'color.sty' is not loaded}%
    \renewcommand\color[2][]{}%
  }%
  \providecommand\transparent[1]{%
    \errmessage{(Inkscape) Transparency is used (non-zero) for the text in Inkscape, but the package 'transparent.sty' is not loaded}%
    \renewcommand\transparent[1]{}%
  }%
  \providecommand\rotatebox[2]{#2}%
  \newcommand*\fsize{\dimexpr\f@size pt\relax}%
  \newcommand*\lineheight[1]{\fontsize{\fsize}{#1\fsize}\selectfont}%
  \ifx\svgwidth\undefined%
    \setlength{\unitlength}{414.33559989bp}%
    \ifx\svgscale\undefined%
      \relax%
    \else%
      \setlength{\unitlength}{\unitlength * \real{\svgscale}}%
    \fi%
  \else%
    \setlength{\unitlength}{\svgwidth}%
  \fi%
  \global\let\svgwidth\undefined%
  \global\let\svgscale\undefined%
  \makeatother%
  \begin{picture}(1,0.47474913)%
    \lineheight{1}%
    \setlength\tabcolsep{0pt}%
    \put(0,0){\includegraphics[width=\unitlength,page=1]{figure28case3_1AB.pdf}}%
    \put(0.13893542,0.13627403){\color[rgb]{0,0,0}\makebox(0,0)[lt]{\lineheight{0}\smash{\begin{tabular}[t]{l}$\gamma$\end{tabular}}}}%
    \put(0.04411351,0.06006685){\color[rgb]{0,0,0}\makebox(0,0)[lt]{\lineheight{0}\smash{\begin{tabular}[t]{l}$B$\end{tabular}}}}%
    \put(0.42434811,0.24714126){\color[rgb]{0,0,0}\makebox(0,0)[lt]{\lineheight{0}\smash{\begin{tabular}[t]{l}$A$\end{tabular}}}}%
    \put(-0.00055624,0.25438177){\color[rgb]{0,0,0}\makebox(0,0)[lt]{\lineheight{0}\smash{\begin{tabular}[t]{l}$D\!=\!Q_{s+1}$\end{tabular}}}}%
    \put(0,0){\includegraphics[width=\unitlength,page=2]{figure28case3_1AB.pdf}}%
    \put(0.33571051,0.18348868){\color[rgb]{0,0,0}\makebox(0,0)[lt]{\lineheight{0}\smash{\begin{tabular}[t]{l}$\gamma_{AD}^L$\end{tabular}}}}%
    \put(0.25116928,0.21419977){\color[rgb]{0,0,0}\makebox(0,0)[lt]{\lineheight{0}\smash{\begin{tabular}[t]{l}$Q_{s+t+2}$\end{tabular}}}}%
    \put(0,0){\includegraphics[width=\unitlength,page=3]{figure28case3_1AB.pdf}}%
    \put(0.57979464,0.25724663){\color[rgb]{0,0,0}\makebox(0,0)[lt]{\lineheight{0}\smash{\begin{tabular}[t]{l}$B$\end{tabular}}}}%
    \put(0,0){\includegraphics[width=\unitlength,page=4]{figure28case3_1AB.pdf}}%
    \put(0.68386246,0.3430178){\color[rgb]{0,0,0}\makebox(0,0)[lt]{\lineheight{0}\smash{\begin{tabular}[t]{l}$\gamma_1$\end{tabular}}}}%
    \put(0.76719547,0.4589859){\color[rgb]{0,0,0}\makebox(0,0)[lt]{\lineheight{0}\smash{\begin{tabular}[t]{l}$\gamma_{AD}^R$\end{tabular}}}}%
    \put(0,0){\includegraphics[width=\unitlength,page=5]{figure28case3_1AB.pdf}}%
    \put(0.57617438,0.00473391){\color[rgb]{0,0,0}\makebox(0,0)[lt]{\lineheight{0}\smash{\begin{tabular}[t]{l}$B$\end{tabular}}}}%
    \put(0.95640902,0.19180837){\color[rgb]{0,0,0}\makebox(0,0)[lt]{\lineheight{0}\smash{\begin{tabular}[t]{l}$A$\end{tabular}}}}%
    \put(0,0){\includegraphics[width=\unitlength,page=6]{figure28case3_1AB.pdf}}%
    \put(0.66986414,0.08048911){\color[rgb]{0,0,0}\makebox(0,0)[lt]{\lineheight{0}\smash{\begin{tabular}[t]{l}$\beta$\end{tabular}}}}%
    \put(0.7426833,0.20194789){\color[rgb]{0,0,0}\makebox(0,0)[lt]{\lineheight{0}\smash{\begin{tabular}[t]{l}$\epsilon$\end{tabular}}}}%
    \put(0,0){\includegraphics[width=\unitlength,page=7]{figure28case3_1AB.pdf}}%
    \put(0.95811253,0.4436212){\color[rgb]{0,0,0}\makebox(0,0)[lt]{\lineheight{0}\smash{\begin{tabular}[t]{l}$A$\end{tabular}}}}%
    \put(0.1134456,0.25656051){\makebox(0,0)[lt]{\lineheight{1.25}\smash{\begin{tabular}[t]{l}$Q_s\!=\!Q_{s+2}$\end{tabular}}}}%
    \put(0.56674717,0.44077458){\makebox(0,0)[lt]{\lineheight{1.25}\smash{\begin{tabular}[t]{l}$D$\end{tabular}}}}%
  \end{picture}%
\endgroup%

%% file: conjectures.pdf_tex
%% Creator: Inkscape 1.0 (4035a4f, 2020-05-01), www.inkscape.org
%% PDF/EPS/PS + LaTeX output extension by Johan Engelen, 2010
%% Accompanies image file 'conjectures.pdf' (pdf, eps, ps)
%%
%% To include the image in your LaTeX document, write
%%   \input{<filename>.pdf_tex}
%%  instead of
%%   \includegraphics{<filename>.pdf}
%% To scale the image, write
%%   \def\svgwidth{<desired width>}
%%   \input{<filename>.pdf_tex}
%%  instead of
%%   \includegraphics[width=<desired width>]{<filename>.pdf}
%%
%% Images with a different path to the parent latex file can
%% be accessed with the `import' package (which may need to be
%% installed) using
%%   \usepackage{import}
%% in the preamble, and then including the image with
%%   \import{<path to file>}{<filename>.pdf_tex}
%% Alternatively, one can specify
%%   \graphicspath{{<path to file>/}}
%% 
%% For more information, please see info/svg-inkscape on CTAN:
%%   http://tug.ctan.org/tex-archive/info/svg-inkscape
%%
\begingroup%
  \makeatletter%
  \providecommand\color[2][]{%
    \errmessage{(Inkscape) Color is used for the text in Inkscape, but the package 'color.sty' is not loaded}%
    \renewcommand\color[2][]{}%
  }%
  \providecommand\transparent[1]{%
    \errmessage{(Inkscape) Transparency is used (non-zero) for the text in Inkscape, but the package 'transparent.sty' is not loaded}%
    \renewcommand\transparent[1]{}%
  }%
  \providecommand\rotatebox[2]{#2}%
  \newcommand*\fsize{\dimexpr\f@size pt\relax}%
  \newcommand*\lineheight[1]{\fontsize{\fsize}{#1\fsize}\selectfont}%
  \ifx\svgwidth\undefined%
    \setlength{\unitlength}{314.97892904bp}%
    \ifx\svgscale\undefined%
      \relax%
    \else%
      \setlength{\unitlength}{\unitlength * \real{\svgscale}}%
    \fi%
  \else%
    \setlength{\unitlength}{\svgwidth}%
  \fi%
  \global\let\svgwidth\undefined%
  \global\let\svgscale\undefined%
  \makeatother%
  \begin{picture}(1,0.74331401)%
    \lineheight{1}%
    \setlength\tabcolsep{0pt}%
    \put(0,0){\includegraphics[width=\unitlength,page=1]{conjectures.pdf}}%
    \put(-0.00040305,0){\makebox(0,0)[lt]{\lineheight{1.25}\smash{\begin{tabular}[t]{l}$A(0,0)$\end{tabular}}}}%
    \put(0.85643658,0.72){\makebox(0,0)[lt]{\lineheight{1.25}\smash{\begin{tabular}[t]{l}$B(q,p)$\end{tabular}}}}%
    \put(0.97358619,0.72){\makebox(0,0)[lt]{\lineheight{1.25}\smash{\begin{tabular}[t]{l}$C(q+1,p)$\end{tabular}}}}%
    \put(0.97739809,0.59614073){\makebox(0,0)[lt]{\lineheight{1.25}\smash{\begin{tabular}[t]{l}$D(q+1,p-1)$\end{tabular}}}}%
    \put(0.2936367,0){\makebox(0,0)[lt]{\lineheight{1.25}\smash{\begin{tabular}[t]{l}$E(q-p+1,0)$\end{tabular}}}}%
  \end{picture}%
\endgroup%